\documentclass[11pt, one side, article]{memoir}

\settrims{0pt}{0pt} 
\settypeblocksize{*}{36.1pc}{*} 
\setlrmargins{*}{*}{1} 
\setulmarginsandblock{.98in}{.98in}{*} 
\setheadfoot{\onelineskip}{2\onelineskip} 
\setheaderspaces{*}{1.5\onelineskip}{*} 
\checkandfixthelayout

\usepackage{amsthm}
\usepackage{mathtools}

\usepackage[inline]{enumitem}
\usepackage{ifthen}
\usepackage[utf8]{inputenc} 
\usepackage{xcolor}

\usepackage[backend=biber, backref=true, maxbibnames = 10, style = alphabetic]{biblatex}
\usepackage[bookmarks=true, colorlinks=true, linkcolor=blue!50!black,
citecolor=orange!50!black, urlcolor=orange!50!black, pdfencoding=unicode]{hyperref}
\usepackage[capitalize]{cleveref}

\usepackage{tikz}

\usepackage{amssymb}
\usepackage{newpxtext}
\usepackage[varg,bigdelims]{newpxmath}
\usepackage{mathrsfs}
\usepackage{dutchcal}
\usepackage{fontawesome}
\usepackage{ebproof}
\usepackage{stmaryrd}
\usepackage{float}

  \crefformat{enumi}{\card#2#1#3}
  \crefalias{chapter}{section}

  \hypersetup{final}

  \setlist{nosep}
  \setlistdepth{6}

  \usetikzlibrary{ 
  	cd,
  	math,
  	decorations.markings,
		decorations.pathreplacing,
  	positioning,
  	arrows.meta,
  }
  
\tikzset{
	tick/.style={postaction={
  	decorate,
    decoration={markings, mark=at position 0.5 with
    	{\draw[-] (0,.4ex) -- (0,-.4ex);}}}
  }
} 
\tikzset{
	slash/.style={postaction={
  	decorate,
    decoration={markings, mark=at position 0.5 with
    	{\node[font=\footnotesize] {\rotatebox{90}{$\approx$}};}}}
  }
}

\theoremstyle{definition}
\newtheorem{definitionx}{Definition}[chapter]
\newtheorem{examplex}[definitionx]{Example}
\newtheorem{remarkx}[definitionx]{Remark}
\newtheorem{notation}[definitionx]{Notation}

\theoremstyle{plain}

\newtheorem{theorem}[definitionx]{Theorem}
\newtheorem{proposition}[definitionx]{Proposition}
\newtheorem{corollary}[definitionx]{Corollary}
\newtheorem{lemma}[definitionx]{Lemma}

\newtheorem*{theorem*}{Theorem}
\newtheorem*{proposition*}{Proposition}
\newtheorem*{corollary*}{Corollary}
\newtheorem*{lemma*}{Lemma}
\newtheorem*{warning*}{Warning}

\newenvironment{example}
  {\pushQED{\qed}\examplex}
  {\popQED\endexamplex}
  
 \newenvironment{remark}
  {\pushQED{\qed}\remarkx}
  {\popQED\endremarkx}
  
  \newenvironment{definition}
  {\pushQED{\qed}\definitionx}
  {\popQED\enddefinitionx}

\DeclareSymbolFont{stmry}{U}{stmry}{m}{n}
\DeclareMathSymbol\fatsemi\mathop{stmry}{"23}

\DeclareFontFamily{U}{mathx}{\hyphenchar\font45}
\DeclareFontShape{U}{mathx}{m}{n}{
      <5> <6> <7> <8> <9> <10>
      <10.95> <12> <14.4> <17.28> <20.74> <24.88>
      mathx10
      }{}
\DeclareSymbolFont{mathx}{U}{mathx}{m}{n}
\DeclareFontSubstitution{U}{mathx}{m}{n}
\DeclareMathAccent{\widecheck}{0}{mathx}{"71}

\DeclareMathOperator{\Hom}{Hom}

\DeclareMathOperator*{\colim}{colim}

\DeclareMathOperator{\ob}{Ob}

\DeclareMathOperator{\StrFunn}{StrFun}

\newcommand{\distpb}{{\color{blue}\lrcorner}}

\newcommand{\cat}[1]{\mathcal{#1}}
\newcommand{\Cat}[1]{\mathbf{#1}}

\newcommand{\id}{\mathrm{id}}

\newcommand{\too}{\longrightarrow}
\newcommand{\tto}{\rightrightarrows}
\newcommand{\To}[2][]{\xrightarrow[#1]{#2}}

\newcommand{\Tto}[3][13pt]{\begin{tikzcd}[cramped, ampersand replacement=\&, text height=1ex, text depth=.3ex, column sep=#1]\ar[r, shift left=4pt, "#2"]\ar[r, shift right=4pt, "#3"']\&{}\end{tikzcd}}
\newcommand{\Too}[1]{\xrightarrow{\;\;#1\;\;}}
\newcommand{\from}{\leftarrow}

\newcommand{\From}[1]{\xleftarrow{#1}}

\newcommand{\imp}{\Rightarrow}

\newcommand{\card}{\,^{\sharp}}

\newcommand{\op}{^\tn{op}}

\newcommand{\tn}[1]{\textnormal{#1}}

\newcommand{\smset}{\Cat{Set}}

\newcommand{\smcat}{\Cat{Cat}}
\newcommand{\lgcat}{\Cat{CAT}}
\newcommand{\prof}{\mathbb{P}\Cat{rof}}

\newcommand{\ccatsharp}{\mathbb{C}\Cat{at}^{\sharp}}

\newcommand{\sspan}{\mathbb{S}\Cat{pan}}

\newcommand{\yon}{\mathcal{y}}
\newcommand{\poly}{\Cat{Poly}}

\newcommand{\tri}{\mathbin{\triangleleft}}

\newcommand{\restrict}[1]{\big|_{#1}}

\newcommand{\qqand}{\qquad\text{and}\qquad}

\newcommand{\coalg}{\tn{-}\Cat{Coalg}}

\newcommand{\bic}[2]{{}_{#1}\Cat{Comod}_{#2}}

\newcommand{\E}{{\cat{E}}}

\renewcommand{\P}{{\Cat{P}}}
\newcommand{\trii}{\tri}
\newcommand{\comon}{\Cat{Comon}}
\newcommand{\Eh}{{\widehat{\E}}}
\newcommand{\A}{{\cat{A}}}
\newcommand{\B}{{\cat{B}}}
\newcommand{\Ah}{{\widehat{\A}}}
\newcommand{\CC}{{\mathcal{C}}}
\newcommand{\DD}{{\mathcal{D}}}
\renewcommand{\C}{\mathbf{C}}
\newcommand{\D}{\mathbf{D}}

\newcommand{\biglens}[2]{
     \begin{bmatrix}{\vphantom{f_f^f}#2} \\ {\vphantom{f_f^f}#1} \end{bmatrix}
}
\newcommand{\littlelens}[2]{
     \begin{bsmallmatrix}{\vphantom{f}#2} \\ {\vphantom{f}#1} \end{bsmallmatrix}
}
\newcommand{\lens}[2]{
  \relax\if@display
     \biglens{#1}{#2}
  \else
     \littlelens{#1}{#2}
  \fi
}

\newcommand{\pb}[2]{\prescript{\vphantom{X}}{#1\vphantom{X}}{\times}^{\vphantom{X}}_{#2\vphantom{X}}}

\newcommand{\frowny}[1]{\overset{#1}{\frown}}

\newcommand{\disc}{\tn{disc}}

\allowdisplaybreaks
\setsecnumdepth{section}
\settocdepth{section}
\begin{document}

\title{Structures on Categories of Polynomials}

\author{Brandon T. Shapiro \and David I. Spivak}

\date{\vspace{-.2in}}

\maketitle

\begin{abstract}
We define the monoidal category $(\poly_\E,\yon,\tri)$ of polynomials under composition in any category $\E$ with finite limits, including both cartesian and vertical morphisms of polynomials, and generalize to this setting the Dirichlet tensor product of polynomials $\otimes$, duoidality of $\otimes$ and $\tri$, closure of $\otimes$, and coclosures of $\tri$. We also prove that $\tri$-comonoids in $\poly_\E$ are precisely the internal categories in $\E$ whose source morphism is exponentiable, generalizing a result of Ahman-Uustalu equating categories with polynomial comonads, and show that coalgebras in this setting correspond to internal copresheaves. Finally, the double category of ``typed'' polynomials in $\E$ is recovered using $\tri$-bicomodules in $\poly_\E$.
\end{abstract}

\setcounter{tocdepth}{1}
\begin{KeepFromToc}
\tableofcontents
\end{KeepFromToc}

\chapter{Introduction}

Polynomial functors, which in the simplest case are endofunctors $F\colon\smset\to\smset$ that send a set $X$ to the set
\[
\sum_{I \in P} X^{p[I]}
\]
for some set $P$ and some family of sets $(p[I])_{I \in P}$, have been used in mathematics to study topics ranging across algebra, topology, dynamics, computer science, and even category theory itself. They have been generalized by various authors such as in \cite{kock2012polynomial,weber2015polynomials}\footnote{While our focus will be on the approach taken in these papers by Gambino--Kock and Weber respectively, other generalizations can be found at \href{https://topos.site/p-func-workshop/2021/}{https://topos.site/p-func-workshop/2021/} and \href{https://topos.site/p-func-workshop/}{https://topos.site/p-func-workshop/}} to, for instance, include multiple variables and go between categories other than that of sets. 

In more recent work by Ahman--Uustalu \cite{ahman2016directed}, Garner, and Spivak \cite{spivak2021functorial}, the category $\poly=\poly_\smset$ of single-variable polynomial endofunctors on the category of sets has been shown to host an enormous range of categorical structure and models for various concepts in category theory and applications. For instance, it has many different monoidal structures including but not limited to addition, multiplication, and composition of polynomials; many pairs of these monoidal structures form duoidal or distributive structures on the category; many of these monoidal structures have closures and/or coclosures; comonoids with respect to composition are precisely categories and their homomorphisms are cofunctors; and bicomodules between comonoids correspond to parametric right adjoint functors between presheaf categories.

The aim of this paper is to extend these additional structures and results to the category $\poly_\E$ of polynomials in any category $\E$ with finite limits. To do so, we first define $\poly_\E$ by extending the construction of the category of polynomials and morphisms between them from \cite{kock2012polynomial} to the more general setting introduced in \cite{weber2015polynomials} which considered only cartesian morphisms of polynomials. Unlike both Gambino--Kock and Weber's approaches, we consider only single-variable polynomials which form the objects of a monoidal category (under composition) rather than the 1-cells of a bicategory or horizontal morphisms of a double category, so as to be able to work with them in the style of \cite{ahman2016directed,spivak2021functorial} and other work by the second author. This requires the additional assumption of finite products in the category $\E$ that are not needed in \cite{weber2015polynomials}, but which are nonetheless present in most examples of interest and necessary for the definition of additional structures on $\poly_\E$ such as the Dirichlet tensor product (\cref{tensordef}). We also recover in \cref{typedpolybicomod} and \cref{typedpolycompose} the double category of multivariable polynomials in $\E$ within the framework of bicomodules in $\poly_\E$, so that no structure from \cite{kock2012polynomial,weber2015polynomials} is left behind by taking this approach.

A primary motivation for this work is to generalize these structures and results from $\poly$ to the setting of polynomials in $\smcat$, which have been studied by various authors such as in \cite{weber2015polynomials,weber2015operads,shapiro2022thesis} and provide an elegant formalism for constructions in category theory such as free (co)product completions, free lax (co)limit completions, and wreath products of categories. In this setting, \cref{comonoidcategory} shows that comonoids for the composition monoidal structure on $\poly_{\smcat}$ are precisely strict double categories whose source functor is exponentiable, so that double categories as well as ordinary categories can be studied by polynomial methods. Likewise, as discussed in \cref{doublecopresheaves}, in this setting \cref{coalgcopresheaves} demonstrates that coalgebras for polynomial comonads on $\smcat$ are precisely the standard notion of double copresheaves (\cite{pare2011yoneda}) on the transpose of the associated double category, recovering by a basic construction for polynomials a fundamental concept from double category theory.

Both in defining the monoidal category of polynomials in a finite limit category $\E$ and proving \cref{comonoidcategory} we make extensive use of the technique of translating potentially difficult proofs in a general finite limit category $\E$ to more straightforward proofs in a presheaf category. The arguments in the more general setting typically involve large diagrams of morphisms and pullback squares which can be particularly illuminating and are included wherever it is practical to do so, but in many instances these diagrams are too large and tedious to be easily understood or fit within the bounds of a page. By contrast, the same proofs in a presheaf category can typically rely on the same element-wise reasoning available in the category of sets and as such are often much easier to follow. This technique is first introduced in \cref{sec.presheaves} and relies on the theory of dense functors.

In the interest of brevity, we limit the scope of this paper to structures in the category $\poly_\E$ that require only finite limits in the category $\E$, though we will occasionally allow related additional assumptions such as $\E$ being cartesian closed. In particular we do not assume that $\E$ has coproducts and as such the usual addition and multiplication operations on $\poly$ are not available. In future work, we plan to thoroughly explore the additional structure present in $\poly_\E$ when $\E$ is a lextensive category, i.e.\ one with finite limits and coproducts that interact nicely with pullbacks.

\section*{Notation}

We write 1 for the terminal object in any category, $\smset$ for the category of sets, $\smcat$ for the category of small categories, and $\lgcat$ for the category of locally small categories.

In a category $\E$ with an object $A$, we write $\E/A$ for the overcategory at $A$, whose objects are morphisms in $\E$ with codomain $A$ and whose morphisms are commuting triangles.

For a category $\A$, we write $\Ah$ for the category $\smset^{\A\op}$ of presheaves on $\A$. For $X$ a presheaf on $\A$ and $a$ an object in $\A$, we write $X_a$ for the set $X(a)$. We write $\yon(a)\coloneqq\A(-,a)$ to denote the presheaf represented by $a$.

Given maps $A\To{f}C\From{g}B$, we denote the pullback as either $A\pb{f}{g}B$ or $A \times_C B$.


\section*{Acknowledgments}

This material is based upon work supported by the Air Force Office of Scientific Research under award number FA9550-20-1-0348.

Nate Soares and Jesse Liptrap first pointed out to us a counterexample to exponentiability of the right coclosure, though we chose to describe a different one in \cref{coclosurecounterexample}. Bryce Clarke provided many helpful comments and references.

\chapter{Categories with Pullbacks}\label{chap.pb}

We begin by reviewing pullback functors between overcategories and their left and right adjoints, with an emphasis on the example of presheaf categories and functors which preserve these structures.

\section{Adjoints to pullback}

\begin{definition}\label{sigmadelta}
For any morphism $f \colon B \to A$ in a category $\E$ with pullbacks, $\Delta_f$ denotes the pullback functor $\E/A \to \E/B$ and $\Sigma_f$ denotes the postcomposition functor $\E/B \to \E/A$.
\end{definition}

Note that $\Sigma$ is functorial in the sense that the assignments $A \mapsto \E/A$ and $f \mapsto \Sigma_f$ form a functor $\E \to \lgcat$, and analogously $\Delta$ forms a pseudofunctor $\E\op \to \lgcat$.

\begin{lemma}
For any morphism $f \colon B \to A$ in a category $\E$ with pullbacks, the functor $\Sigma_f$ is left adjoint to $\Delta_f$.
\end{lemma}

\begin{proof}
By the universal property of the pullback, for any morphisms $C \to A$ and $D \to B$, maps $D \to C$ over $A$ correspond bijectively with maps $D \to \Delta_f C$ over $B$.
\end{proof}

\begin{definition}\label{pi}
A morphism $f \colon B \to A$ in a category $\E$ with pullbacks is \emph{exponentiable} if the functor $\Delta_f \colon \E/A \to \E/B$ also has a right adjoint, in which case we denote it by $\Pi_f \colon \E/B \to \E/A$.
\end{definition}

When $f$ is exponentiable, for an object $g \colon C \to B$ in the category $\E/B$, the counit $\epsilon$ of the adjunction has the form $\Delta_f\Pi_f C \to C$ over $B$, as shown in the diagram \eqref{eqn.pidist}.
\begin{equation}\label{eqn.pidist}
\begin{tikzcd}[row sep=small]
&[-20pt] \Delta_f \Pi_f C \rar{h} \ar[dl, "\epsilon"'] \ar[phantom]{ddr}[pos=0]{\distpb} &[-5pt] \Pi_f C \ar[dd, "\Pi_fg"] \\
C \ar{dr}[swap]{g} \\
& B \rar[swap]{f} & A
\end{tikzcd}
\end{equation}
By definition, $\Pi_f C$ is defined by the universal property that for all objects $D$ over $A$ in $\E$, 
\begin{equation}
\Hom_{\E/A}\left(D,\Pi_f C\right) \cong \Hom_{\E/B}\left(\Delta_f D, C\right),
\end{equation}
or equivalently for arbitrary objects $D$ in $\E$,
\begin{equation}\label{eqn.pihoms}
\Hom_{\E}\left(D,\Pi_f C\right) \cong \coprod_{D \to A} \Hom_{\E/B}\left(\Delta_f D, C\right),
\end{equation}
naturally in $D$ in both descriptions. This shows that any pullback of $f$ whose projection to $B$ factors through $C$ arises uniquely as a pullback of $h$; in other words, the pullback square in \eqref{eqn.pidist} is a \emph{distributivity pullback} around $(f,g)$, terminal among pullbacks of $f$ factoring through $g$ in this manner \cite[Definitions 2.2.1 -- 2.2.2]{weber2015polynomials}.

\begin{example}\label{ex.iso_distributivity}
If $f$ (resp.\ $g$) is an isomorphism then $\Pi_fC=C$ and $\Pi_fg=g$ (resp.\ $\Pi_fC=A$ and $\Pi_fg=\id_A$) forms a distributivity pullback.
\end{example}

\begin{lemma}[{\cite[Lemmas 2.2.4 and 2.2.5]{weber2015polynomials}}]\label{comppb}
Exponentiable morphisms in a category $\E$ with pullbacks are closed under composition and under pullback along morphisms in $\E$.
\end{lemma}

\begin{definition}
For $\E$ and $\E'$ both categories with pullbacks, a pullback-preserving functor $F \colon \E \to \E'$ \emph{preserves exponentials} if for every exponentiable morphism $f \colon B \to A$ in $\E$, $F(f) \colon F(B) \to F(A)$ is exponentiable in $\E$ and the diagram in \eqref{eqn.functorpi} commutes.
\begin{equation}\label{eqn.functorpi}
\begin{tikzcd}[column sep=large]
\E/B \rar{\Pi_f} \dar[swap]{F} & \E/A \dar{F} \\
\E'/F(B) \rar[swap]{\Pi_{F(f)}} & \E'/F(A)
\end{tikzcd}
\end{equation} 
A pullback-preserving functor $F \colon \E \to \E'$ \emph{reflects exponentials} when for any diagram as in \eqref{eqn.pbaround} in $\E$ with the morphism $f$ exponentiable, 
\begin{equation}\label{eqn.pbaround}
\begin{tikzcd}[row sep=small]
&[-15pt] \Delta_f D \rar \ar{dl} \ar[phantom]{ddr}[pos=0]{\lrcorner} & D \ar{dd} \\
C \ar{dr}[swap]{g} \\
& A \rar[swap]{f} & B
\end{tikzcd}
\end{equation}
if $F(D) \cong \Pi_{F(f)} F(C)$ then $D \cong \Pi_f C$.
\end{definition}

Functors which preserve and reflect exponentials will be useful for transforming questions about arbitrary categories with pullbacks into questions about particularly nice categories with pullbacks which are thereby easier to answer.

\section{Presheaf categories and singular functors}\label{sec.presheaves}

The nice categories with pullbacks for our purposes will be presheaf categories. For any category $\A$, its presheaf category $\Ah \coloneq \smset^{\A\op}$ is \emph{locally cartesian closed}, meaning that every morphism is exponentiable. In fact, in $\Ah$ each right adjoint $\Pi_f$ has an explicit formula.

\begin{lemma}\label{presheafpi}
Given morphisms $g \colon Z \to Y$ and $f \colon Y \to X$ in $\Ah$, there is a morphism $\Pi_f g \colon \Pi_f Z \to X$ where for each object $a$ of $\A$, 
\begin{equation}\label{eqn.presheafpi}
\left(\Pi_f Z\right)_a \cong \coprod_{\sigma \colon \yon(a) \to X} \Hom_{\Ah/Y}\left(\Delta_\sigma Y, Z\right)
\end{equation}
and the component of the map $\Pi_f g$ at $A$ is the projection map sending the pair $\left(\sigma,\gamma\right)$ to the element $\sigma \in \Hom_\Ah\left(\yon(a), X\right) \cong X_a$.
\end{lemma}

\begin{proof}
The defining universal property of $\Pi_f$ in \eqref{eqn.pihoms} is that 
\begin{equation}\label{eqn.presheafpiW}
\Hom_{\Ah}\left(W,\Pi_f Z\right) \cong \coprod_{\sigma \colon W \to X} \Hom_{\Ah/Y}\left(\Delta_\sigma Y, Z\right),
\end{equation}
naturally with respect to the presheaf $W$. Letting $W = \yon(a)$ in \eqref{eqn.presheafpiW} shows that if $\Pi_f Z$ exists, it must have the form in \eqref{eqn.presheafpi}. It then only remains to deduce \eqref{eqn.presheafpiW} from the definition in \eqref{eqn.presheafpi}.

With $\Pi_f Z$ as defined in \eqref{eqn.presheafpi}, a map $W \to \Pi_f Z$ amounts to natural functions $W_a \to X_a \cong \Hom_\Ah\left(\yon(a), X\right)$ for each $a$ in $\A$, which we summarize by a map $\sigma \cong W \to X$, along with maps in $\Hom_{\Ah/Y}\left(\Delta_{\sigma \circ w} Y, Z\right)$ natural in $w \colon \yon(a) \to W$. The latter natural transformation can be written as an element of the set
\begin{align*}
  \lim_{w \colon \yon(a) \to W} \Hom_{\Ah/Y}\left(\Delta_{\sigma \circ w} Y, Z\right)&\cong 
  \Hom_{\Ah/Y}\left(\colim_{w \colon \yon(a) \to W}\left(\yon(a) \pb{\sigma \circ w}{f} Y\right), Z\right)\\&\cong 
  \Hom_{\Ah/Y}\left(\left(\colim_{w \colon \yon(a) \to W}\yon(a)\right) \pb{\sigma \circ w}{f}Y, Z\right)\\&\cong 
  \Hom_{\Ah/Y}\left(W \pb{\sigma}{f} Y, Z\right)\\&= 
  \Hom_{\Ah/Y}\left(\Delta_\sigma Y, Z\right),
\end{align*}
where the above isomorphisms arise from contravariant $\Hom$ functors sending colimits to limits, colimits commuting with products in one variable in the presheaf category $\Ah/X$, and the standard colimit decomposition of a presheaf. This completes the proof.
\end{proof}

While not all categories $\E$ with pullbacks allow for such convenient formulas for their exponentials, we will consider functors from $\E$ to a presheaf category $\Ah$ which preserve and reflect exponentials, so that we can work in $\Ah$ instead of $\E$.

In particular, recall that a functor $F \colon \A \to \E$ induces a limit-preserving functor $F^\ast \colon \E \to \Ah$, often called a \emph{singular functor}, defined on objects $E$ in $\E$ by 
\[
(F^\ast E)_a = \Hom_\E(F(a), E).
\]

\begin{definition}\label{def.dense}
A functor $F \colon \A \to \E$ is \emph{dense} if for all objects $E$ in $\E$, the natural map
\[
\colim\left(F/E \to \A \To{F} \E\right)\To{\cong}E
\]
is an isomorphism
\end{definition}

In particular, if the functor $F$ is dense then the induced functor $F^\ast \colon \E \to \Ah$ is fully faithful \cite[Lemma 1.7]{ulmer1968dense}. 

\begin{example}
The identity functor $\id_\E \colon \E \to \E$ is dense as $\id_\E/A \cong \E/A$ has a terminal object, and the corresponding fully faithful functor $\E \to \Eh$ is the Yoneda embedding.
\end{example}

While for any $\E$ we could always choose the identity as our dense functor, so long as we allow ourselves to consider sizes larger than that of $\E$ in our universe hierarchy, in many cases of interest there are more efficient choices of the category $\A$ and functor $F \colon \A \to \E$ that do not require $\Ah$ to be such a larger category than $\E$. 

\begin{example}
There is a dense functor from the simplex category $\Delta$, whose objects are finite nonempty ordinals and morphisms are monotone maps, to the category of small categories sending each ordinal to the corresponding poset category. The associated singular functor $\smcat \to \widehat\Delta$ is the notoriously fully faithful \emph{nerve} functor, sending a category $\CC$ to the simplicial set whose $n$-simplices are given by the strings of $n$ composable morphisms in $\CC$.
\end{example}

As mentioned above, the functor $F^\ast$ always preserves limits. When $F$ is fully faithful and dense, $F^\ast$ also in fact preserves and reflects exponentials.

\begin{proposition}\label{densepi}
For any fully faithful dense functor $F \colon \A \to \E$, the induced fully faithful functor $F^\ast \colon \E \to \Ah$ preserves and reflects exponentials.
\end{proposition}

\begin{proof}
Given morphisms $f \colon B \to A$ and $C \to B$ in $\E$, recall (similarly to \eqref{eqn.presheafpiW}) that $\Pi_f C$ is defined by the universal property 
\[
\Hom_\E\left(D, \Pi_f C\right) \cong \coprod_{\sigma \colon D \to A} \Hom_{\E/B}\left(D \times_A B, C\right),
\]
naturally in the object $D$ of $\E$. By this and the definition of the functor $F^\ast$, we have that for each object $a$ in $\A$ there are isomorphisms
\begin{align}\label{eqn.pipreserve}
\Hom_\Ah\left(\yon(a), F^\ast \Pi_f C\right) &\cong 
\Hom_\E\left(F(a), \Pi_f C\right)\\\nonumber&\cong
\coprod_{\sigma \colon F(a) \to A} \Hom_{\E/B}\left(F(a) \times_A B, C\right).
\intertext{As $F^\ast$ is fully faithful and pullback-preserving, we continue}
	\nonumber&\cong
\coprod_{\sigma \colon F^\ast F(a) \to F^\ast A} \Hom_{\Ah/F^\ast B}\left(F^\ast F(a) \times_{F^\ast A} F^\ast B, F^\ast C\right)\\\nonumber&\cong 
\Hom_\Ah\left(F^\ast F(a), \Pi_{F^\ast f} F^\ast C\right)\\\nonumber&\cong \Hom_\Ah\left(\yon(a), \Pi_{F^\ast f} F^\ast C\right),
\end{align}
where $F^\ast F(a)$ is precisely the representable presheaf $\yon(a)$ in $\Ah$ as $F$ is fully faithful. By the Yoneda lemma then, $F^\ast \Pi_f C$ and $\Pi_{F^\ast f} F^\ast C$ are isomorphic and so $F^\ast$ preserves exponentials.

To see that $F^\ast$ reflects exponentials, let $E$ be an object of $\E$ such that $F^\ast E \cong \Pi_{F^\ast f} F^\ast C$. We then have isomorphisms natural in $D$, 
\begin{align*}
\Hom_\E\left(D,E\right) \cong \Hom_\Ah\left(F^\ast D, F^\ast E\right) &\cong \coprod_{\sigma \colon F^\ast D \to F^\ast A} \Hom_{\Ah/F^\ast B}\left(F^\ast D \times_{F^\ast A} F^\ast B, F^\ast C\right),
\intertext{As $F^\ast$ is fully faithful and preserves pullbacks, we complete the proof that $E \cong \Pi_f C$:}
&\cong\coprod\limits_{\sigma \colon D \to A} \Hom_{\E/B}\left(D \times_A B, C\right)\cong\Hom_{\E}(D,\Pi_f C)
\end{align*}
\end{proof}

\begin{remark}
Note that the proof that $F^\ast$ reflects exponentials relies only on the fact that $F^\ast$ is fully faithful and pullback-preserving, while the proof that $F^\ast$ preserves exponentials relies on the particulars of the definition of $F^\ast$. However, $F^\ast$ is often a right adjoint (specifically when $\E$ has certain colimits), in which case its preservation of exponentials follows from being a fully faithful right adjoint.
\end{remark}

\chapter{Polynomials and Polynomial Functors}\label{chap.poly}

We now set out to define a monoidal category $\poly_\E$ for a category $\E$ with finite limits, whose objects are \emph{polynomials in $\E$}.

\begin{definition}\label{polynomial}
For $\E$ a category with pullbacks, a polynomial $p$ in $\E$ is an exponentiable morphism $P_\ast \To{p} P$ in $\E$. 
\end{definition}

\begin{notation}
Unlike in \cref{chap.pb} where exponentiable morphisms were denoted in the style of $f \colon B \to A$ to establish their basic properties, when regarded as polynomials we use the convention of denoting the codomain of $p$ by its capitalization $P$ and the domain by $P_\ast$. This is meant to reduce the number of letters in the namespace when working with many polynomials at once, and emphasize that the polynomial $p$ is the primary object of study while its domain and codomain as a morphism in $\E$ are subsidiary components of $p$. While the symbol $(-)_\ast$ can temporarily be regarded as merely syntax for the domain, in \cref{pstar} we define an operation $(-)_\ast$ on polynomials such that $P_\ast$ is the codomain of a polynomial $p_\ast$.
\end{notation}

The relationship between exponentiable morphisms $p$ and the classical notion of polynomials is based on the \emph{polynomial functor} associated to $p$ (see \cref{polynomialfunctor}), particularly in the case when $\E$ is the category of sets.

\begin{example}\label{setpoly}
In the category of sets, which is a topos and hence locally cartesian closed, every function $P_\ast \To{p} P$ is an exponentiable morphism. For an element $I \in P$, let $p[I]$ denote its preimage in $P_\ast$. The polynomial functor associated to $p$ is the endofunctor on $\smset$ sending a set $X$ to the set
\[
\sum_{I \in P} X^{p[I]},
\]
where $\sum$ denotes an indexed disjoint union and $X^{p[I]}$ is the set of functions from the set $p[I]$ to $X$. Accordingly, we sometimes denote $p$ by
\[
\sum_{I \in P} \yon^{p[I]}.
\]

This endofunctor bears resemblance to a classical polynomial function, as it has the form of a sum of powers of $X$. While absent from this notation, analogues of natural number coefficients in classical polynomials would arise by grouping the summands according to the cardinality of their exponent; for instance, $Y \times X^Z \cong \sum_{I \in Y} X^Z$.
\end{example}

We will sometimes describe the codomain $P$ as the \emph{positions} of a polynomial $p$ and the domain $P_\ast$ as the \emph{directions}, consistent with the terminology in \cite[Definition 2.1.1]{spivak2021functorial}. The objects $P$ and $P_\ast$ will also occasionally be referred to as the \emph{base} and \emph{total space} of $p$, respectively.

\begin{example}\label{presheafpoly}
Every morphism in a presheaf category $\Ah$ is also exponentiable. For $P_\ast \To{p} P$ any morphism in $\Ah$, we can similarly define its fibers $p[x]$ for each object $a$ in $\A$ and $x \in P_a$ as the pullback of $p$ along $x \colon \yon(a) \to P$. As the pullback functor $\Delta_p$ is a right adjoint and thus preserves pullbacks, we have
\[
P_\ast \cong \Delta_p P \cong \Delta_p \colim_{x \colon \yon(a) \to P} \yon(a) \cong \colim_{x \colon \yon(a) \to P} p[x].\qedhere
\]
\end{example}

\begin{example}\label{catpoly}
In the category $\smcat$ of small categories, an exponentiable morphism $P_\ast \To{p} P$ is a \emph{Conduch\'e functor}.\footnote{While the name ``Conduch\'e functor'' is in most common use, they were first defined by Giraud (see \cite{street2001powerful}).} A functor is Conduch\'e if for any composite morphism $f = f' \circ f''$ in $P$ and morphism $g$ in $P_\ast$ such that $p(g) = f$, there exists a factorization $g = g' \circ g''$ such that $p(g')=f'$, $p(g'')=f''$, and this factorization is unique up to ``morphisms of factorizations.'' A morphism of factorizations from $g'_1,g''_1$ to $g'_2,g''_2$ is a morphism $h$ in $P_\ast$ such that $g'_1 = g'_2 \circ h$ and $g''_2 = h \circ g''_1$. 

Conduch\'e functors can be thought of as the functors entirely determined by their fibers; we write $p[I]$ for the pullback of $p$ along the inclusion of a single object $I$ in $P$, and $p[f]$ for the pullback of $p$ along the inclusion of a single morphism $f$ in $P$. For $f \colon I \to J$ in $P$, the cospan $p[I] \to p[f] \from p[J]$ is the collage of a profunctor from $p[I]$ to $p[J]$, and the Conduch\'e condition is equivalent to the assertion that for $f = f' \circ f''$ as above, the profunctor associated to $p[f]$ is isomorphic to the composite of the profunctors $p[f']$ and $p[f'']$. Therefore, a Conduch\'e functor into $P$ can be equivalently regarded as a pseudofunctor $p[-] \colon P \to \prof$ to the pseudo-double category of categories, functors, profunctors, and maps of profunctors (see \cite{street2001powerful} for a more thorough account).\footnote{Here the ordinary category $P$ is by default regarded as a double category whose horizontal category is $P$ and whose vertical category is discrete.} 
\end{example}

The definition of a polynomial in \cref{polynomial} is slightly different from the polynomials described in \cite{kock2012polynomial,weber2015polynomials}, which we instead call \emph{typed polynomials}.

\begin{definition}\label{typedpoly}
A typed polynomial in the category $\E$ is a diagram of the form 
\[
B \from P_\ast \To{p} P \to A,
\] 
where the morphism $p$ is exponentiable. For fixed objects $A$ and $B$, this is also called a \emph{polynomial from $B$ to $A$}.
\end{definition}

In \cite{kock2012polynomial,weber2015polynomials}, a polynomial from $B$ to $A$ is regarded as a 1-cell in a double category or bicategory whose objects are those of $\E$. When the category $\E$ has a terminal object 1, typed polynomials from 1 to 1 are precisely the (untyped) polynomials of \cref{polynomial}.

We will instead describe a monoidal category $\poly_\E$ in which polynomials are the objects. While this is technically a weaker result than the analogous constructions in \cite{kock2012polynomial,weber2015polynomials}, neglecting the bicategory structure of typed polynomials and requiring (only for the monoidal structure) the category $\E$ to have all finite limits rather than merely pullbacks, we have found the language of monoidal categories to be particularly convenient for describing various structures on and applications of polynomials (for instance, in \cite{spivak2021functorial,shapiro2022duoidal,shapiro2022dynamic} and \cref{chap.structure,chap.comonoid} below). That said, this choice is purely aesthetic: all of the arguments we make in this chapter for the monoidal category of polynomials in a finitely complete category $\E$ apply equally well to the double category of polynomials in a category $\E$ with pullbacks. Moreover, in \cref{typedpolycompose} we recover Gambino-Kock's double category and Weber's bicategory of polynomials via bicomodules in our monoidal category $\poly_\E$.

While the construction presented here of $\poly_\E$ for a finitely complete category $\E$ fills a gap in the literature, as \cite{kock2012polynomial} restricts the category $\E$ to be locally cartesian closed and \cite{weber2015polynomials} includes only cartesian morphisms between polynomials (\cref{morphism}), the main contribution of this chapter is a novel technique for checking the axioms of this monoidal category. This approach reduces those axioms to checking the case when $\E$ is a presheaf category, where the result is known by \cite{kock2012polynomial}. While this approach does not significantly differ from that of \cite{kock2012polynomial}, both using the data of a fully faithful strong monoidal functor out of $\poly_\E$ to verify that $\poly_\E$ is in fact a monoidal (resp. bi-) category, it sets the stage for using the same techniques for proving results about $\poly_\E$ in subsequent sections where this reduction significantly simplifies the proofs.

\section{The category of polynomials}

Unlike morphisms in the arrow category or twisted arrow category which correspond to commuting squares and factorizations, morphisms in the category of polynomials represent forwards-pointing maps on positions and backwards-pointing maps on the directions, made precise using pullbacks. 

\begin{definition}\label{morphism}
Given polynomials $p$ and $q$ in a category $\E$ with pullbacks, a morphism of polynomials from $p$ to $q$ is an isomorphism class of diagrams as in \eqref{eqn.polymap}, 
\begin{equation}\label{eqn.polymap}
\begin{tikzcd}
P_\ast \dar[swap]{p} & P \times_Q Q_\ast \lar[swap]{\phi^\sharp} \rar \dar \ar[phantom]{dr}[pos=0]{\lrcorner} &[-10pt] Q_\ast \dar{q} \\
P & P \lar[equals] \ar[r, "\phi_1"'] & Q
\end{tikzcd}
\end{equation}
where the relevant isomorphisms are between choices of the pullback $P \times_Q Q_\ast$.

A morphism $\phi$ is called \emph{cartesian} when $\phi^\sharp$ is an isomorphism (hence the identity in some representative diagram), and \emph{vertical} when $\phi_1$ is the identity.
\end{definition}

Given a morphism $(\phi_1,\phi^\sharp)$ of polynomials, we will sometimes call $\phi_1$ the component on positions, and $\phi^\sharp$ the component on directions.

\begin{example}
Given polynomials $p$ and $q$ in $\smset$, a morphism $\phi \colon p \to q$ consists of a function $\phi_1 \colon P \to Q$ along with for each element $I \in P$ a function $q[\phi_1I] \to p[I]$. When presented in this manner these morphisms of polynomials do not require any isomorphism classes as the pullback $P \times_Q Q_\ast$ can be canonically chosen to be 
\[
\coprod_{I \in P} q[\phi_1I].
\]
Moreover, it is clear how to compose morphisms of this form, and the evident composition of the relevant functions does in fact agree with the general form of composing morphisms of polynomials defined below in \cref{tentativepoly}.
\end{example}

\begin{example}
Given polynomials $p$ and $q$ in $\smcat$, a morphism $\phi \colon p \to q$ amounts to a functor $\phi_1 \colon P \to Q$ along with for each object $I \in P$ a functor $\phi^\sharp_I \colon q[\phi_1I] \to p[I]$ and for each morphism $f \colon I \to J$ in $P$ a square
\[
\begin{tikzcd}
{q[\phi_1I]} \rar[tick,""{name=U, below}]{{q[\phi_1f]}} \dar[swap]{\phi^\sharp_I} & {q[\phi_1J]} \dar{\phi^\sharp_J} \\
{p[I]} \rar[tick,swap,""{name=V, above}]{{p[f]}} & {p[J]}
\arrow[Rightarrow,shorten=4,from=U,to=V]
\end{tikzcd}
\]
in the double category $\prof$. When regarding $p$ and $q$ as pseudofunctors to $\prof$ as in \cref{catpoly}, this is precisely the data of a colax triangle
\[
\begin{tikzcd}
P \ar{rr}{\phi_1} \ar{dr}[swap]{{p[-]}} & & Q \ar{dl}{{q[-]}} \ar[Rightarrow, shorten=28, shift left=6]{ll}[swap]{\phi^\sharp} \\ 
& \prof
\end{tikzcd}
\]
in the 2-category of pseudo-double categories, pseudo-double functors, and vertical transformations. From this perspective, composing these morphisms is entirely routine, and indeed composition of colax triangles over $\prof$ corresponds to the composition of morphisms of polynomials in \cref{tentativepoly}.
\end{example}

We now set about proving that these polynomials and their morphisms form a category. While it would be straightforward to do so directly using only the universal property of pullbacks, we instead demonstrate the technique of reducing questions about polynomials in $\E$ to questions about polynomials in a presheaf category $\Ah$ where the results of \cite{kock2012polynomial} apply. To do so, we introduce some terminology useful for transfering properties of one categorical structure to another. 

\begin{remark}
The approach we take here is precisely the technique used in \cite{kock2012polynomial} to establish the bicategory of polynomials in a locally cartesian closed category. While their technique could apply almost verbatim to this setting of polynomials in a category with finite limits as well, it is much more straightforward to instead rely on their result and transfer structure from a category of polynomials in a presheaf category $\Ah$, rather than all the way from a category of endofunctors on $\E$. However, the construction in \cite[2.7]{kock2012polynomial} of the natural transformation of polynomial functors associated to a vertical morphism of polynomials relies critically on the underlying category being locally cartesian closed. Our approach avoids this construction, and in \cref{sec.polynomialfunctors} we discuss a more general construction of this transformation which does not require local cartesian closure.
\end{remark}

\begin{definition}
A \emph{tentative category} consists of a collection of objects, a set of morphisms for each ordered pair of such objects, and choices of identity and composite morphisms resembling those of a category but without necessarily satisfying the unitality and associativity equations. 

A \emph{tentative functor} from one tentative category to another is a mapping from the objects and morphisms of the first to the objects and morphisms of the second which preserves sources, targets, identities, and compositions.
\end{definition}

\begin{example}
There is a forgetful functor from categories to tentative categories. We will refer to tentative categories in the image of this functor as \emph{established} categories. In practice, all of the tentative categories (or other tentative structures) in this paper will ultimately be established, but will make use of tentative functors to prove it, hence the need for a notion of functor that can be defined without assuming the domain and/or codomain are yet established categories.
\end{example}

Our motivating example of a tentative category is that of polynomials and morphisms between them in a category $\E$ with pullbacks. We will ultimately prove it to be a category, but not before defining a tentative functor out of it.

\begin{definition}\label{tentativepoly}
For $\E$ a category with pullbacks, let $\poly_\E$ be the tentative category which has
\begin{itemize}
	\item as objects, polynomials in $\E$;
	\item as morphisms, morphisms of polynomials in $\E$;
	\item as the identity morphism on a polynomial $P_\ast \To{p} P$, the isomorphism class $\id_p$ represented by the diagram in \eqref{eqn.polyid}; and
\begin{equation}\label{eqn.polyid}
\begin{tikzcd}
P_\ast \dar[swap]{p} & P_\ast \lar[equals] \rar[equals] \dar \ar[phantom]{dr}[pos=0]{\lrcorner} & P_\ast \dar{p} \\
P & P \lar[equals] \rar[equals] & P
\end{tikzcd}
\end{equation}
	\item as the composite of morphisms represented by $(\phi_1,\phi^\sharp)$ and $(\psi_1,\psi^\sharp)$ as in \eqref{eqn.polymapadj},
\begin{equation}\label{eqn.polymapadj}
\begin{tikzcd}
P_\ast \dar[swap]{p} & P \times_Q Q_\ast \lar[swap]{\phi^\sharp} \rar \dar \ar[phantom]{dr}[pos=0]{\lrcorner} &[-10pt] Q_\ast \dar{q} & Q \times_R R_\ast \lar[swap]{\psi^\sharp} \rar \dar \ar[phantom]{dr}[pos=0]{\lrcorner} &[-10pt] R_\ast \dar{r} \\
P & P \lar[equals] \ar[r, "\phi_1"'] & Q & Q \lar[equals] \ar[r, "\psi_1"'] & R
\end{tikzcd}
\end{equation}
the morphism $\psi \circ \phi$ depicted in \eqref{eqn.polymapcomp},
\begin{equation}\label{eqn.polymapcomp}
\begin{tikzcd}[column sep=40pt]
P_\ast \dar[swap]{p} &[15pt] P \times_R R_\ast \lar[swap]{\phi^\sharp \circ (P \times_Q \psi^\sharp)} \rar \dar \ar[phantom]{dr}[pos=0]{\lrcorner} &[-10pt] R_\ast \dar{r} \\
P & P \lar[equals] \ar[r, "\psi_1 \circ \phi_1"'] & R
\end{tikzcd}
\end{equation}
where $P \times_Q Q_\ast \From{P \times_Q \psi^\sharp} P \times_Q (Q \times_R R_\ast) \cong P \times_R R_\ast$ is the map induced on pullbacks by $\psi^\sharp$.
\end{itemize}
\end{definition}

Our proof that the tentative category $\poly_\E$ is in fact a category is based on the following meta-theorem, which says that a tentative structure (category or otherwise) equipped with a fully faithful (in an appropriate sense) tentative structure-preserving map to an established structure is itself established. While we will use the general idea of this meta-theorem repeatedly, for the purpose of this section it takes the following form.

\begin{proposition}\label{categoryestablish}
If a tentative category $\CC$ is equipped with a faithful tentative functor $G$ to an established category $\DD$, then $\CC$ is itself established as a category and $G$ is established as a functor.
\end{proposition} 

\begin{proof}
It is sufficient to show that composition in the tentative category $\CC$ is unital and associative. Let $f,g,h$ be adjacent morphisms in $\C$; we have by functoriality of $G$ and associativity in $\DD$ that 
\[
G((f \circ g) \circ h) = (G(f) \circ G(g)) \circ G(h) = G(f) \circ (G(g) \circ G(h)) = G(f \circ (g \circ h)),
\]
which by faithfulness of $G$ implies that $(f \circ g) \circ h = f \circ (g \circ h)$ and therefore composition in $\C$ is associative. Similarly we have
\[
G(\id \circ f) = G(\id) \circ G(f) = \id \circ G(f) = G(f),
\]
and the entirely analogous right unitality equation, completing the proof.
\end{proof}

We can now prove that $\poly_\E$ is in fact a category.

\begin{theorem}\label{polycategory}
For a category $\E$ with pullbacks, $\poly_\E$ is a category.
\end{theorem}

\begin{proof}
By the discussion in and preceding \cite[2.14]{kock2012polynomial}, this result holds for the case when $\E$ is a presheaf category, where in their notation $\poly_\E$ is the category $\poly_\E(1,1)$. Let $F \colon \A \to \E$ be any dense functor (which if necessary can always be taken to be the identity), and consider the associated singular functor $F^\ast \colon \E \to \Ah$ which is fully faithful and pullback preserving (see page~\pageref{def.dense}).

Using \cref{categoryestablish}, it suffices to construct a faithful tentative functor $\poly_{F^\ast}$ from the tentative category $\poly_\E$ to the established category $\poly_\Ah$. On objects, $\poly_{F^\ast}$ sends a polynomial $p$ in $\E$ to the polynomial $F^\ast(p)$ in $\Ah$. On morphisms, $\poly_{F^\ast}$ acts by applying $F^\ast$ to each arrow in the diagrams as in \eqref{eqn.polymap}, which produces a morphism as $F^\ast$ is a functor and preserves pullbacks. Also by functoriality and pullback-preservation of $F^\ast$, the mapping $\poly_{F^\ast}$ preserves identities and composites. Finally as the functor $F^\ast$ is fully faithful so is $\poly_{F^\ast}$, completing the proof.
\end{proof}

\begin{remark}
The category $\poly_\E$ is isomorphic to a full subcategory of the category of dependent lenses in $\E$ as in \cite[Example 3.5]{spivak2019generalized}, whose objects are all morphisms in $\E$ and whose morphisms are also diagrams as in \eqref{eqn.polymap}. While allowing such additional objects does not interfere with the formation of a category, the monoidal structure defined in the next section required the objects to be exponentiable morphisms in $\E$.
\end{remark}

Before moving on to the monoidal structure on $\poly_\E$, we give a description of certain limits in $\poly_\E$ which will be helpful later on.

\begin{proposition}\label{hasconnectedlimits}
For $\E$ a category with finite connected limits, $\poly_\E$ has limits of finite connected diagrams in which all morphisms are cartesian (\cref{morphism}). Moreover, for $\B$ a connected category, $q_{(-)} \colon \B \to \poly_\E$ a diagram with all morphisms cartesian, and $i_0$ any object in $B$, the limit $\lim_i q_i$ is isomorphic to the polynomial
\[
(Q_{i_0})_\ast \times_{Q_{i_0}} \lim_i Q_i \to \lim_i Q_i.
\]
\end{proposition}

\begin{proof}
First note that this morphism is exponentiable, as it is a pullback of $q_i$ (\cref{comppb}). To see that this construction is independent of the choice of $i_0$, choose any other object $i_1$ in the category $\B$ along with a finite zigzag of morphisms in $\B$ from $i_0$ to $i_1$, which is always possible as $\B$ is connected. By induction on the length of this zigzag, it suffices to assume that there is either a morphism $i_0 \to i_1$ or a morphism $i_0 \from i_1$ in $\B$. 

In the first case, as the morphisms in the diagram are cartesian we have that $(Q_{i_0})_\ast \cong (Q_{i_1})_\ast \times_{Q_{i_1}} Q_{i_0}$, and in the second case we similarly have $(Q_{i_1})_\ast \cong (Q_{i_0})_\ast \times_{Q_{i_0}} Q_{i_1}$. Either way, composition of pullback squares shows that
\[
(Q_{i_0})_\ast \times_{Q_{i_0}} \lim_i Q_i \cong (Q_{i_1})_\ast \times_{Q_{i_1}} \lim_i Q_i.
\]

Now to see that this is a limit in $\poly_\E$, consider a polynomial $p$ with a cone $\phi_i \colon p \to q_i$ over $q_{(-)}$. By the universal property of limits in the category $\E$, the morphisms $P \to Q_i$ induce a map $P \to \lim_i Q_i$, whose pullback along $\lim_i q_i$ is given by
\[
P \times_{\lim_i Q_i} \lim_i Q_i \times_{Q_{i_0}} (Q_{{i_0}})_\ast \cong P \times_{Q_{i_0}} (Q_{i_0})_\ast
\]
for any object $i_0$ in the category $\B$. We therefore have a morphism of polynomials $p \to \lim_i q_i$, where the component on directions is given by that of $\phi_{i_0}$. This definition is independent of the choice of object $i_0$ as the morphisms of polynomials $\phi_i$ commute with the cartesian morphisms in the diagram $q_{(-)}$, and is straightforwardly checked to be unique by the universal properties of the limit $\lim_i Q_i$ and its pullbacks in $\E$.
\end{proof}

\section{Composition of polynomials}\label{sec.composition}

Classical polynomials can be composed with one another as functions; one way to generalize this is by regarding them as polynomial functors, as we discuss in \cref{sec.polynomialfunctors}, but just as composition of classical polynomials admits a formula using the distributivity of products over sums, composition of polynomials in a category $\E$ can be defined using distributivity pullbacks.\footnote{This analogy indeed motivates the name ``distributivity square'' and is made precise by the fact that distributivity pullbacks describe how to swap the application order of functors of the form $\Sigma_f$ and $\Pi_g$ as in \cref{chap.pb}.}

\begin{definition}\label{tri}
For polynomials $p$ and $q$ in a category $\E$ with finite limits, the \emph{composition product} $p \trii q$ is given by the composite of the top row of morphisms in \eqref{eqn.comp}. 
\begin{equation}\label{eqn.comp}
\begin{tikzcd}
(\Pi_p (Q \times P_\ast) \times_P P_\ast) \times_Q Q_\ast \ar{dd} \ar{rr} \ar[phantom]{ddrr}[pos=0]{\lrcorner} &[-20pt] &[-10pt] \Pi_p (Q \times P_\ast) \times_P P_\ast \ar{rr} \dar \ar[phantom]{ddrr}[pos=0]{\distpb} &[-10pt] & \Pi_p (Q \times P_\ast) \ar{dd} \\
& & Q \times P_\ast \ar{dl} \ar{dr} \\
Q_\ast \rar{q} & Q & {} & P_\ast \rar{p} & P 
\end{tikzcd}
\end{equation}
Note that $p \trii q$ is exponentiable as a composite of pullbacks of exponentiable morphisms $p$ and $q$ (\cref{comppb}).
\end{definition}

\begin{definition}\label{linear}
The \emph{identity polynomial} $\yon$ in a category $\E$ with finite limits is given by the identity map on the terminal object $1$. 
\end{definition}

\begin{example}\label{scalarmult}
For any polynomial $q$ in $\E$ and object $A$, we write $Aq$ for the polynomial $A \times Q_\ast \To{\id_A \times q} A \times Q$, which is exponentiable as a pullback of $q$. By \cref{ex.iso_distributivity}, $Aq \cong A\yon \tri q$, as $\Pi_{\id_A}$ is isomorphic to the identity functor, where $A\yon$ is simply the identity morphism on $A$.
\end{example}

\begin{example}\label{presheafcomp}
Using the notation for polynomials in a presheaf category $\Ah$ from \cref{presheafpoly} and the formula for exponentials in \cref{presheafpi}, we can compute that the polynomial $p \tri q$ in $\Ah$ has as its component at the object $a$ in $\A$ the dependent projection function
\[
 \coprod_{x \in P_a} \coprod_{f \colon p[x] \to Q} \coprod_{y \in p[x]_a} q[f(y)]_a \to
 \coprod_{x \in P_a} \coprod_{f \colon p[x] \to Q}1.\qedhere
\]
\end{example}

\begin{example}\label{catcomp}
For polynomials $p,q$ in $\smcat$, their composite $p \tri q$ has as its base the category with objects pairs of the form
\[
(I \in \ob(P), J \colon p[I] \to Q)
\]
and morphisms $(I,J) \to (I',J')$ pairs of the form
\[
(f \colon I \to I', g \colon p[f] \to Q),
\]
where here $p[f]$ denotes the corresponding collage category and the restrictions of $g$ to $p[I]$ and $p[I']$ agree with $J$ and $J'$ respectively. Because we have $p[f' \circ f] \cong p[f'] \circ p[f]$ as profunctors for any morphism $f' \colon I' \to I''$ in $P$, it is straightforward to define composition of these morphisms.

The fiber of the composite at $(I,J)$ is then given by the total space of the polynomial corresponding to the pseudofunctor
\[
p[I] \To{J} Q \To{q} \prof,
\]
with the fibers on morphisms constructed similarly.
\end{example}

We now describe tentative monoidal categories and functors, to prove that $\yon$ and $\tri$ form a monoidal structure on the category $\poly_\E$ using the same technique as in \cref{polycategory}.

\begin{definition}
A \emph{tentative monoidal category} is a tentative category $\CC$ equipped with a distinguished object $I$ and a function $\otimes \colon \ob(\CC) \times \ob(\CC) \to \ob(\CC)$.

A \emph{tentative strong (resp.\ lax) monoidal functor} is a tentative functor $G \colon \CC \to \DD$ between tentative monoidal categories equipped with identitor and productor isomorphisms (resp.\ morphisms) in $\DD$ of the forms
\[
\iota \colon I_\DD \to G(I_\CC) \qqand \psi_{a,b} \colon G(a) \otimes G(b) \to G(a \otimes b)
\]
respectively, with the productor defined for each pair of objects $a,b$ in the tentative category $\CC$.
\end{definition}

As with tentative categories, monoidal categories are included among tentative monoidal categories and can be called \emph{established}.

\begin{example}
For $\E$ a category with finite limits, $(\poly_\E,\yon,\tri)$ is a tentative monoidal category. 
\end{example}

\begin{proposition}\label{monoidalestablish}
If a tentative monoidal category $\CC$ is equipped with a fully faithful tentative strong monoidal functor $G$ to an established monoidal category, $\DD$, then $\CC$ is established as a monoidal category and $G$ is established as a monoidal functor.
\end{proposition}

\begin{proof}
By \cref{categoryestablish}, $\CC$ is a category, so it remains only to define the monoidal stucture and show that $G$ is a monoidal functor. This requires defining the action of $\otimes$ on morphisms of $\CC$, the associator isomorphisms $\alpha$, and the unitor isomorphisms $\lambda,\rho$, and proving that these isomorphisms are natural and the triangle and pentagon equations hold in $\CC$, that the associativity and unitality equations hold for the tentative monoidal functor $G$, and that its productors assemble into a natural transformation.

For morphisms $f \colon a \to a'$ and $g \colon b \to b'$ in $\CC$, there is a morphism
\[
G(a \otimes b) \To{\psi_{a,b}^{-1}} G(a) \otimes G(b) \To{G(f) \otimes G(g)} G(a') \otimes G(b') \To{\psi_{a',b'}} G(a' \otimes b')
\]
in the monoidal category $\DD$. As $G$ is fully faithful, this morphism determines a unique morphism $f \otimes g \colon a \otimes b \to a' \otimes b'$ in $\CC$ making the productor $\rho$ a natural transformation.

For objects $a,b,c$ in $\CC$, there is a morphism
\begin{multline*}
G((a \otimes b) \otimes c) \To{\psi_{a \otimes b,c}^{-1}} G(a \otimes b) \otimes G(c)
\To{\psi_{a,b}^{-1} \otimes \id_{G(c)}} (G(a) \otimes G(b)) \otimes G(c)\\
\Too{\alpha} G(a) \otimes (G(b) \otimes G(c)) \To{\id_{G(a)} \otimes \psi_{b,c}} G(a) \otimes G(b \otimes c) \To{\psi_{a,b \otimes c}} G(a \otimes (b \otimes c))
\end{multline*}
in $\DD$. As $G$ is fully faithful, this determines a unique morphism $\alpha_{a,b,c} \colon (a \otimes b) \otimes c \to a \otimes (b \otimes c)$ in $\CC$ satisfying the associativity equation for $G$. The left and right unitors can be similarly defined as the unique morphisms in $\CC$ making the left and right unitality equations for $G$ hold.

Finally, as $G$ is functorial and faithful, and the triangle, pentagon, and naturality diagrams for $\alpha,\lambda,\rho$ commute in $\DD$, the analogous diagrams must also commute in $\CC$ as their images commute in $\DD$ and there is a unique morphism in $\CC$ sent by $G$ to the total composite of each diagram.
\end{proof}

\cref{monoidalestablish} reduces the task of proving that $(\poly_\E,\yon,\tri)$ is a monoidal category to defining a tentative strong monoidal functor to an established monoidal category, for which as usual we will use a singular functor.

\begin{theorem}\label{monoidalcat}
When $\E$ is a category with finite limits, $(\poly_\E,\yon,\tri)$ is a monoidal category.
\end{theorem}

\begin{proof}
Let $F \colon \A \to \E$ be a fully faithful dense functor (which can always be taken to be the identity, if need be) and $F^\ast \colon \E \to \Ah$ its associated fully faithful singular functor. As $F^\ast$ preserves finite limits as well as exponentials by \cref{densepi}, the functor $\poly_{F^\ast}$ from \cref{polycategory} preserves both the polynomial $\yon$ and the composition operation $\tri$ up to isomorphism, the former defined using only the terminal object and the latter defined using only products, exponentials, and pullbacks. This makes $\poly_{F^\ast} \colon \poly_\E \to \poly_\Ah$ a tentative monoidal functor. Since $\poly_\Ah$ is an established monoidal category by \cite[2.16]{kock2012polynomial} (where it is instead denoted $\poly_\Ah(1,1)$), the result follows from \cref{monoidalestablish}.
\end{proof}

\section{Polynomial Functors}\label{sec.polynomialfunctors}

So far, our proofs have not made explicit use of the polynomial functor associated to a polynomial, as it was made standard in \cite{kock2012polynomial,weber2015polynomials}; however, we have done so implicitly, as our proofs have relied on the results of \cite{kock2012polynomial} which themselves rely on polynomial functors. Later in \cref{sec.coalgebras} we will need to make explicit use of polynomial functors, so we recall relevant aspects of theory here. 

\begin{definition}\label{polynomialfunctor}
For a polynomial $P_\ast \To{p} P$ in a category $\E$ with finite limits, its associated \emph{polynomial endofunctor} $\P(p) \colon \E \to \E$ is defined as the composite
\[
\E \cong \E/1 \To{\Delta_!} \E/P_\ast \To{\Pi_p} \E/P \To{\Sigma_!} \E/1 \cong \E,
\]
where $!$ refers to the unique map to the terminal object 1. In particular, $\Delta_! \colon \E \to \E/P_\ast$ is the functor $- \times P_\ast$ and $\Sigma_! \colon \E/P \to \E$ is the usual forgetful functor.
\end{definition}

Note that for each $p$, the functor $\P(p)$ preserves connected limits, being the composite of a right adjoint with a forgetful functor $\E/P \to \E$, both of which preserve connected limits.

By the universal properties defining the functors $\Pi_p$ and $\Delta_!$ (products), for an object $A$ in $\E$ the object $\P(p)(A)$ has the universal property that a morphism $B \to \P(p)(A)$ is uniquely determined by morphisms $B \to P$ and $B \times_P P_\ast \to A$.

\begin{example}
As discussed in \cref{setpoly}, polynomials in the category of sets induce polynomial functors of the form 
\[
X\mapsto\sum_{I \in P} X^{p[I]}
\]
for $p[I]$ the preimage of $I$ in $P_\ast$. To see this, given a set $X$ we have that the set $\Delta_!X = X \times P_\ast$ in $\smset/P_\ast$. By \cref{presheafpi}, $\Pi_p\Delta_!X$ is the set of pairs $(I,f \colon p[I] \to X \times P_\ast)$ where $\pi_2 \circ f$ is the inclusion $p[I] \to P_\ast$, meaning $f$ amounts to simply a function $p[I] \to X$. Finally applying $\Sigma_!$ forgets that this set maps to $P$, which can be interpreted as taking a $P$-indexed disjoint union of the sets $X^{p[I]}$.
\end{example}

\begin{example}
Observe that for a polynomial $p$ and an object $A$ in $\E$, $\P(p)(A)$ is the base of the polynomial $p \tri A\yon$. This means that in $\smcat$, following \cref{catcomp}, $\P(p)(A)$ has objects of the form $(I \in P, f \colon p[I] \to A)$ and morphisms of the form $(h \colon I \to J, H \colon p[h] \to A)$. 

This construction is in fact a generalization of the \emph{wreath product} of categories \cite[Section 4]{cisinski2011theta}. For a category $B$ (often the simplex category $\Delta$) equipped with a functor $\gamma$ to Segal's category $\Gamma$ (dual to the skeleton of finite sets and partial functions) and any category $C$, the wreath product is a category $B \wr C$ whose objects are of the form $(b \in \ob(B); c_1,...,c_{\gamma(b)} \in \ob(C))$ and whose morphisms are arrangements of arrows in $C$ with sources and targets determined by the reverse partial functions in $\Gamma$. 

By the inclusion $\Gamma \to \sspan$ of spans of functions whose forward component is monic, and the inclusion $\sspan \to \prof$ of profunctors between discrete categories, a functor $B \to \Gamma$ determines a polynomial whose corresponding polynomial endofunctor on $\smcat$ is precisely $B \wr -$. Polynomial composition similarly generalizes the wreath product for functors $B \to \sspan$ developed in \cite{shapiro2022thesis}.
\end{example}

This assignment also extends to morphisms, though using a slightly different construction than that of \cite{kock2012polynomial}.

\begin{lemma}\label{nattrans}
Given a morphism of polynomials $\phi \colon p \to q$ in a category $\E$ with finite limits, there is a corresponding natural transformation $\P(\phi) \colon \P(p) \to \P(q)$.
\end{lemma}

\begin{proof}
When $\phi$ is cartesian this transformation is constructed precisely as in \cite[2.1]{kock2012polynomial} for the case when $I=J=1$, but as mentioned previously the construction for vertical morphisms in \cite[2.7]{kock2012polynomial} does not apply here as the morphism $\phi^\sharp \colon Q_\ast \to P_\ast$ may not be exponentiable. However, this is not a failure of their result in this context but of their particular construction, which uses exponentiation along $\phi^\sharp$ (there denoted $w$) for efficiency rather than necessity. 

We can instead construct this transformation by the diagram in \eqref{eqn.verticaltransformation}, which uses only the adjunctions $\Delta_p \dashv \Pi_p$ and $\Delta_q \dashv \Pi_q$ along with pseudofunctoriality of $\Delta_{(-)}$, noting that as $\phi$ is assumed to be vertical we have that $P=Q$:
\begin{equation}\label{eqn.verticaltransformation}
\begin{tikzcd}
& & \E/P \ar{dr}[description]{\Delta_p} \ar[""{name=U, above}, shift left=1]{ddrr}[description]{\Delta_q} \ar[equals,""{name=S, below}, shift left=.5]{rr} & &[110pt] \E/P \rar{\Sigma_!} \ar[Rightarrow, shorten >=35, shorten <=5, shift right=20]{dd}[pos=.25]{\eta_q} & \E \\
& \E/P_\ast \ar{ur}{\Pi_p} \ar[equals,""{name=T, above}]{rr} & & \E/P_\ast \ar{dr}[swap,description]{\Delta_{\phi^\sharp}} \rar[phantom, pos=.1, shift left=1]{\cong} \dar[phantom, shift right=10]{\cong} & {} \\
\E \ar{ur}{\Delta_!} \ar[""{name=V, above}]{rrrr}[swap]{\Delta_!} & & & {} & \E/Q_\ast \ar{uu}[swap]{\Pi_q}
\arrow[Rightarrow,shorten=4,from=1-3,to=T,"\epsilon_p"]
\end{tikzcd}
\end{equation}

To link up with Gambino-Kock, suppose now that $\phi^\sharp$ happens to be exponentiable. Then $\Pi_q$ factors as $\Pi_q\cong\Pi_p\Pi_{\phi^\sharp}$ and the unit $\eta_q$ of the adjunction $\Delta_q \dashv \Pi_q$ factors as $\eta_q\cong(\Pi_p\eta_{\phi^\sharp}\Delta_p) \circ \eta_p$; regardless of exponentiability, $\Delta_q\cong\Delta_{\phi^\sharp}\Delta_p$. As by the triangle identities $(\Pi_p\epsilon_p) \circ (\eta_p\Pi_p) = \id_{\Pi_p}$, the transformation depicted here agrees with that of \cite[Paragraph 2.7]{kock2012polynomial}, where in their notation $J=I=1$, $A=P$, $B'=P_*$, $B=Q_*$, $f'=p$, $f=q$, and $w=\phi^\sharp$.
\end{proof}

Observe that the assignment $\P$ from \cref{nattrans} respects identities and composites up to isomorphism, by pseudofunctoriality of $\Sigma,\Delta,\Pi$ (which preserve identities up to isomorphism) and the verbatim argument of \cite[Proposition 1.12]{kock2012polynomial}.

\begin{lemma}
For $\E$ a category with finite limits, there is a natural isomorphism $\beta \colon \P(\yon) \cong \id_\E$, and for polynomials $p,q$ there is a natural isomorphism $\alpha_{p,q} \colon \P(p) \circ \P(q) \cong \P(p \trii q)$.
\end{lemma}

This information suffices to prove that $\P$ is a fully faithful monoidal functor, and could also be used via \cref{monoidalestablish} to construct the monoidal category structure on $\poly_\E$ as we have established that $\P$ is a tentative monoidal functor to $\StrFunn(\E,\E)$, which as we now discuss is fully faithful into the monoidal subcategory $\StrFunn(\E,\E)$ of \emph{strong} functors and natural transformations.

\begin{definition}
A \emph{strength} on an endofunctor $F$ on a finite product category $\E$ is a family of morphisms 
\[
\tau_{A,B} \colon A \times F(B) \to F(A \times B)
\]
natural in the objects $A,B$ in $\E$, such that $\tau_{1,B} = \id_{F(B)}$ and $\tau_{A \times A',B} = \tau_{A,A' \times B} \circ (\id_A \times \tau_{A',B})$ up to the coherence isomorphisms for products in $\E$.

A natural transformation $\psi \colon F \imp F'$ between functors equipped with strengths $\tau,\tau'$ respectively is \emph{strong} if for all objects $A,B$ in $\E$, 
\[
\tau'_{A,B} \circ (\id_A \times \psi_B) = \psi_{A \times B} \circ \tau_{A,B}.
\]
We write $\StrFunn(\E,\E)$ for the category of endofunctors on $\E$ equipped with a strength and strong natural transformations between them.
\end{definition}

Note that this definition specializes that of \cite[Definition 1.3]{kock2012polynomial} to the special case of $\E/I$ when $I$ is the terminal object. The following lemma is also described in \cite[Definition 1.3]{kock2012polynomial}, but in order to make this notion concrete we sketch a conceptual proof using composition of polynomials.

\begin{lemma}
For $\E$ a finite limit category, every polynomial endofunctor $\P(p)$ on $\E$ has a strength.
\end{lemma}

\begin{proof}
It is straightforward to check from the definition that for any objects $A,B$ of $\E$, 
\[
A \times \P(p)(B) \cong \P(A\yon \tri p)(B) \qqand \P(p)(A \times B) \cong \P(p \tri A\yon)(B),
\]
so it suffices to provide a morphism of polynomials $A\yon \tri p \to p \tri A\yon$ natural in $A$. Recall from \cref{scalarmult} that $A\yon \tri p \cong Ap$, so that this morphism is cartesian and is induced using the universal property of distributivity squares by the projection maps $A \times P \to P$ and $(A \times P) \times_P P_\ast \cong A \times P_\ast \to A$.
\end{proof}

A strong natural transformation $\P(p) \to \P(q)$ is then a natural transformation commuting with the transformations $\P(A\yon \tri -) \to \P(- \tri A\yon)$ for all objects $A$ in $\E$.

\begin{theorem}\label{Pfunctor}
For $\E$ a category with finite limits, the assignment $p \mapsto \P(p)$ and $\phi \mapsto \P(\phi)$ determines a fully faithful monoidal functor $\P \colon \poly_\E \to \StrFunn(\E,\E)$.
\end{theorem}

\begin{proof}
By the proof of \cref{monoidalcat} and \cite[Theorem 2.17]{kock2012polynomial} there is a fully faithful monoidal functor $\poly_\E \to \poly_\Ah \to \StrFunn(\Ah,\Ah)$ for any fully faithful dense functor $F \colon \A \to \E$, and it agrees with our definition of $\P$ when $E$ is taken to be a presheaf category $\Ah$. 

By \cref{densepi}, the singular functor $F^\ast \colon \E \to \Ah$ preserves and reflects finite limits and exponentials, so when $\P(F^\ast(p))$ is applied to an object of the form $F^\ast(A)$ it produces precisely $F^\ast(\P(p)(A))$. Therefore, this faithful monoidal functor induces another of the form $\P \colon \poly_\E \to \StrFunn(\E,\E)$. 
\end{proof}

\section{Functors between categories of polynomials}

We have thus far made extensive use of the fully faithful monoidal functor $\poly_{F^\ast} \colon \poly_\E \to \poly_\Ah$ induced by a dense functor $\Ah \to \E$ into a category $\E$ with finite limits. We conclude this chapter by summarizing how more general functors between categories with finite limits induce (colax) monoidal functors between their categories of polynomials.

\begin{proposition}\label{polyfun}
Let $\E$ and $\E'$ denote categories with finite limits. A functor $G \colon \E \to \E'$ which preserves pullbacks and exponentiable morphisms induces a colax monoidal functor
\[
	\poly_G \colon \poly_\E \to \poly_{\E'}
\]
whose identitor and productor morphisms are cartesian. 

Furthermore $\poly_{(-)}$ forms a 2-functor from the 2-category of finite limit categories, pullback- and exponentiable morphism-preserving functors, and cartesian natural transformations to the 2-category of monoidal categories, colax monoidal functors, and monoidal natural transformations.
\end{proposition}

\begin{proof}
On objects, $\poly_G$ sends a polynomial $p$ in $\E$ to the morphism $Gp$ in $\E'$, which is exponentiable by assumption. On morphisms, $\poly_G$ simply applies $G$ to each morphism in the diagram \eqref{eqn.polymap}, and the result is a morphism of polynomials since $G$ preserves pullbacks (and it is well defined because $G$ preserves isomorphism classes). 

To see that $\poly_G$ is colax monoidal, we define the colax identitor as the unique cartesian morphism $G(\id_1) = \id_{G(1)} \to \id_1$ in $\poly_{\E'}$. To define the productor, note that applying $G$ to the distributivity pullback around $Q \times P_\ast \to P_\ast \To{p} P$ as in \eqref{eqn.comp} results in a pullback around $G(Q \times P_\ast) \to G(P_\ast) \to G(P)$ and hence (by factorization through the canonical map $G(Q \times P_\ast) \to G(Q) \times G(P_\ast)$) around $G(Q) \times G(P_\ast) \to G(P_\ast) \to G(P)$. We therefore have a diagram as in \eqref{eqn.productor}, whose top row of pullback squares is the cartesian productor morphism $G(p \trii q) \to G(p) \trii G(q)$.
\begin{equation}\label{eqn.productor}
\begin{tikzcd}
\bullet \dar \ar{rr} \ar[phantom]{ddrr}[pos=0]{\lrcorner} & & \bullet \ar{rr} \dar \ar[phantom]{ddrr}[pos=0]{\lrcorner} & & G(\Pi_p (Q \times P_\ast)) \dar \\
\bullet \ar{dd} \ar{rr} \ar[phantom]{ddrr}[pos=0]{\lrcorner} & & \bullet \ar{rr} \dar \ar[phantom]{ddrr}[pos=0]{\distpb} & & \Pi_{G(p)} (G(Q) \times G(P_\ast)) \ar{dd} \\
& & G(Q) \times G(P_\ast) \ar{dl} \ar{dr} & & {} \\
G(Q_\ast) \rar{G(q)} & G(Q) & {} & G(P_\ast) \rar{G(p)} & G(P) 
\end{tikzcd}
\end{equation}
The coherence equations for colax monoidal functors follow straightforwardly from the universal property of distributivity pullbacks.

Finally, it is straightforward to check from the definition that $\poly_{(-)}$ preserves identities and composition of functors, and sends cartesian natural transformations to natural transformations whose components are cartesian morphisms. That these natural transformations $\poly_G \to \poly_{G'}$ are monoidal follows from the assumption that the natural transformations $G \to G'$ are cartesian and the fact that the productors constructed above are cartesian.
\end{proof}

Evidently from the proof of \cref{polyfun}, the identitor of the functor $\poly_G$ is an isomorphism when $G$ preserves terminal objects, and the compositor is an isomorphism when $G$ preserves exponentials. 

\begin{corollary}\label{polyisfunctor}
When a functor $G \colon \E \to \E'$ between categories with finite limits preserves finite limits and exponentials, the induced functor $\poly_G \colon \poly_\E \to \poly_{\E'}$ is monoidal. 
\end{corollary}

\chapter{More Structure on $\poly_\E$}\label{chap.structure}

The category $\poly_\E$ for a category $\E$ with finite limits has a wide range of categorical structures, including additional monoidal products, duoidal structures, and monoidal (co)closures.

\section{Tensor product and duoidality}

While the more classical operations on polynomials are addition and multiplication which require the category $\E$ to have coproducts, the \emph{Dirichlet tensor product} of polynomials requires only products. 

\begin{definition}\label{tensordef}
Given polynomials $p$ and $q$ in a category $\E$ with finite limits, their tensor product $p \otimes q$ is given by their product
\[
P_\ast \times Q_\ast \to P \times Q
\]
in $\E$. The tensor product of morphisms is given by applying the product to each component of the diagram in \cref{morphism}.
\end{definition}

We write $p \otimes q$ for this product rather than $p \times q$ because $\otimes$ is not the product in the category $\poly_E$. 

\begin{example}
For polynomials $p,q$ in $\smset$, $p \otimes q$ is given by
\[
\sum_{(I,J) \in P \times Q} \yon^{p[I] \times q[J]},
\]
which resembles the Dirichlet product of classical polynomials.
\end{example}

The morphism $p \otimes q$ is always exponentiable by \cref{productpi} below, and the tensor product of morphisms of polynomials is again a morphism of polynomials as products commute with pullbacks.

\begin{lemma}\label{productpi}
In a category $\E$ with finite limits, if morphisms $p$ and $q$ are exponentiable, then their product written $p \otimes q$ is exponentiable. 
\end{lemma}

\begin{proof}
The product morphism $p \otimes q$ factors as
\[
P_\ast \times Q_\ast \To{p \times \id_{Q_\ast}} P \times Q_\ast \To{\id_P \times q} P \times Q,
\]
where the factors are pullbacks of $p$ and $q$ respectively along projection morphisms in $\E$. This completes the proof as exponentiable morphisms are closed under composition and pullback by \cref{comppb}.
\end{proof}

To analyze the structure on $\poly_\E$ provided by $\otimes$, we first examine more generally how the functor $\poly$ of \cref{polyfun} from finite limit categories and pullback- and exponentiable morphism-preserving functors to monoidal categories and colax monoidal functors interacts with categorical products.

\begin{lemma}\label{polyismonoidal}
The functor $\poly_{(-)}$ is symmetric monoidal with respect to the cartesian monoidal structure on both its domain 
and codomain.
\end{lemma}

\begin{proof}
First we observe that for 1 the terminal category, $\poly_1$ is itself terminal. 

The productor functor $\poly_\E \times \poly_{\E'} \to \poly_{\E \times \E'}$ sends a pair of polynomials $p$ and $p'$ in $\E$ and $\E'$ respectively to the polynomial $(p,p') \colon (P_\ast,P'_\ast) \to (P,P')$ in $\E \times \E'$, and also acts by pairing on morphisms of polynomials. This assignment is well defined as pullbacks in $\E \times \E'$ are computed componentwise, and it is easily checked that distributivity pullbacks are as well, so a morphism $(p,p')$ in $\E \times \E$ is exponentiable if and only if $p$ and $p'$ are in $\E$ and $\E'$ respectively. Moreover, this functor is evidently an isomorphism as any polynomial (resp. morphism of polynomials) in $\E \times \E'$ projects to a polynomial (resp. morphism of polynomials) in both $\E$ and $\E'$, and monoidal as distributivity pullbacks are computed componentwise in $\E \times \E'$. 

These isomorphisms are easily observed to be natural, symmetric, associative, and unital, by basic unwinding of the definition and the principle that a diagram of tuples is the same as a tuple of diagrams.
\end{proof}

The Dirichlet tensor product $\otimes$ has the identity polynomial $\yon$ as its unit, and together they form another monoidal structure on the category $\poly_\E$, which is easy to prove using \cref{polyismonoidal} and the fact that finite products make every finite limit category a monoid in the category of finite limit categories. Furthermore, this argument also efficiently describes the relationship between the tensor product $\otimes$ and the composition product $\tri$.

\begin{theorem}\label{duoidal}
For $\E$ a category with finite limits, the tensor product of polynomials extends to a functor $\otimes \colon \poly_\E \times \poly_\E \to \poly_\E$ such that $(\poly_\E,\yon,\otimes)$ forms a monoidal category. Furthermore, $(\poly_\E,\yon,\otimes,\yon,\tri)$ forms a normal duoidal category. \end{theorem}
\begin{proof}
The product functors $\times_n \colon \E^n \to \E$ for all $n \ge 0$ all preserve pullbacks as products commute with limits, and preserve exponentiable morphisms by \cref{productpi} and the fact that identity morphisms are exponentiable. Therefore, the category $\E$ with its finite product structure forms a monoid in the cartesian monoidal category of finite limit categories with pullback- and exponentiable morphism-preserving functors. By \cref{polyismonoidal}, the functor $\poly$ is symmetric monoidal and hence preserves monoids, so $\poly_\E$ is a monoid in the category of monoidal categories and colax monoidal functors. This makes $(\poly_\E,\yon,\otimes)$ a monoidal category and moreover a duoidal category with respect to the monoidal structure $(\yon,\tri)$.
\end{proof}

In particular there is an \emph{interchange} map of the form
\begin{equation}\label{interchange}
(p_1\tri p_2)\otimes(q_1\tri q_2)\to(p_1\otimes q_1)\tri(p_2\tri q_2).
\end{equation}

\begin{example}\label{setduoidal}
To see what the duoidal structure looks like for polynomials in $\smset$, observe that for such polynomials $p,q,r,s$ we have
\[
(p \tri q) \otimes (r \tri s) \cong \!\!\!
\sum_{\substack{I \in P \\ J \colon p[I] \to Q \\ K \in R \\ L \colon r[K] \to S}} \!\!\!
\yon^{\sum\limits_{\substack{i \in p[I] \\ k \in r[K]}} q[Ji] \times s[Lk]}
\textrm{and}\quad
(p \otimes r) \tri (q \otimes s) \cong \!\!\!\!\!\!\!\!\!\!\!\!
\sum_{\substack{I \in P \\ K \in R \\ M \colon p[I] \times r[K] \to Q \times S}} \!\!\!\!\!\!\!\!\!\!\!\!
\yon^{\sum\limits_{\substack{i \in p[I] \\ k \in r[K]}} q[M(i,k)] \times s[M(i,k)]}
\]
where the duoidal structure map is the cartesian morphism of polynomials sending $(I,J,K,L)$ to $(I,K,J \times L)$.
\end{example}

As both monoidal structures on $\poly_\E$ share a unit $\yon$ and the tensor product $\otimes$ is symmetric, $\poly_\E$ forms a \emph{physical} duoidal category in the sense of \cite[Definition 2.7]{shapiro2022duoidal}. Also by \cref{polyfun} as all of the colax structure maps involved in functors of the form $\poly_G$ are cartesian, the same is true for the duoidal structure maps in $\poly_\E$. These facts allow us to show that $\poly_\E$ furthermore has an additional $n$-ary operation $\boxtimes_n^N$ for every finite poset $N$ with $n$ elements, called a \emph{dependence structure} (\cite[Definitions 3.8,4.1]{shapiro2022duoidal}).

\begin{corollary}
For a finite limit category $\E$, $\poly_\E$ forms a dependence category derived from its physical duoidal category structure $(\yon,\otimes,\tri)$.
\end{corollary}

\begin{proof}
By \cite[Theorem 4.8]{shapiro2022duoidal}, it suffices to show that the category $\poly_\E$ has finite connected limits which are preserved by the monoidal products $\otimes$ and $\tri$. In fact, by inspection of the proof of that theorem it further suffices to restrict our attention to limits of finite connected diagrams consisting only of cartesian morphisms of polynomials, as the duoidal structure maps used in defining the operations $\boxtimes_n^P$ are all cartesian. 

These limits exist in $\poly_\E$ by \cref{hasconnectedlimits}, and are preserved by $\tri$ by \cref{trilimit}. $\otimes$ preserves these limits in $\poly_\E$ as they are computed using limits in the category $\E$ which commute with products in $\E$. 
\end{proof}

\begin{example}
For polynomials $p_1,...,p_n$ in $\smset$ and $N$ a poset with elements $1,...,n$, the polynomial $\boxtimes_n^N(p_1,...,p_n)$ is defined similarly to those in \cref{setduoidal}, whose positions are tuples $(I_1 \in P_1,...,I_n \in P_n)$ where $I_i$ depends on directions from all those $p_j$ for which $j < i$ in the poset $N$. For instance, in $(p \tri q) \otimes (r \tri s)$, $q$ depends on $p$ and $s$ depends on $r$, while in $(p \otimes r) \tri (q \otimes s)$, $q$ and $s$ both depend on $p$ and $r$. See \cite[Example 4.10]{shapiro2022duoidal} for more details.
\end{example}

\section{Closure for $\otimes$}

The monoidal products $\otimes$ and $\tri$ have various closure structures, which contribute to a rich toolbox of categorical structures in the category $\poly_\E$. For instance, when the finite limit category $\E$ is cartesian closed the monoidal structure given by $\otimes$ is closed as well, as we now describe. Recall that in a cartesian closed category the map $! \colon A \to 1$ is always exponentiable, and the object $\Pi_! B$ is denoted $B^A$.

\begin{definition}
For a cartesian closed finite limit category $\E$, and polynomials $p,q$ in $\E$, the polynomial $[p,q]$ is given by the composite of the top row of morphisms in \eqref{eqn.closure}, 
\begin{equation}\label{eqn.closure}
\begin{tikzcd}
&[-30pt] &[-40pt] &[-30pt] Q_\ast \times_Q (q^p \times P) \ar{ddl} \ar{rrr} &[-30pt] &[-20pt] &[-30pt] q^p \times P \ar{dl} \rar &[20pt] q^p \ar{dd} \\
& & & & {} & \bullet \ar{dl} \ar{dr} & & {} \\
& & \bullet \ar{dl} \ar{rr} & & \bullet \ar{ddr} & {} & Q^P \times P \ar{ddl} \rar & Q^P \ar{ddd} \\
& Q_\ast \times P_\ast \ar{dl} \ar{dr} & & {} & & & & {} \\
P_\ast & & Q_\ast \times P \ar{dl} \ar{rrr} & & & Q \times P \ar{dl} \ar{dr} \\
& Q_\ast \ar{rrr} & & {} & Q & & P \rar & 1
\arrow[phantom, from=1-4, to=2-5, pos=.05, "\lrcorner"]
\arrow[phantom, from=1-7, to=2-8, pos=.05, "\distpb"]
\arrow[phantom, from=2-6, to=3-6, near start, "\vee"]
\arrow[phantom, from=3-3, to=4-4, pos=.05, "\distpb"]
\arrow[phantom, from=3-7, to=4-8, pos=.05, "\distpb"]
\arrow[phantom, from=5-3, to=6-4, pos=.05, "\lrcorner"]
\end{tikzcd}
\end{equation}
each of which is exponentiable as a pullback of an exponentiable morphism. Note that the three suitable pentagons in \eqref{eqn.closure} are distributivity pullbacks.
\end{definition}

While the full proof is far too tedious to write here, it can be checked by chasing through the diagram in \eqref{eqn.closure} that this construction of $[p,q]$ is covariantly functorial in $q$ and contravariantly functorial in $p$.\footnote{The interested reader should be cautioned that this fact is most easily checked separately in the four cases of cartesian morphisms into $p$, vertical morphisms into $p$, cartesian morphisms out of $q$, and vertical morphisms out of $q$, the first and last of which being the more involved.}

\begin{example}
In $\poly_{\smset}$, by \cite[Proposition 2.1.11]{spivak2021functorial} the closure polynomial has the form 
\[
[p,q] \cong \sum_{\phi \colon p \to q} \yon^{\sum_{I \in P} q[\phi_1I]} \cong \prod_{I \in P} \sum_{J \in Q} \prod_{j \in q[J]} \sum_{i \in p[I]} \yon.
\]
The form on the left, which is more intuitive as a polynomial whose positions are morphisms from $p$ to $q$, is often more convenient to use in practice, while the form on the right inspires the general construction of the closure in \eqref{eqn.closure} and motivates the choice of directions. In particular, the form on the right constructs using sums and products a polynomial whose positions will be morphisms from $p$ to $q$, and the directions are then determined by the effect of applying this construction to the unit $\yon$ of the monoidal structure. While this alone is sufficient to justify the choice of directions in the form on the left, it should also be noted that given a morphism $\phi \colon p \to q$, a position $I$ in $p$ and a direction $j$ in $q[\phi_1 I]$ determines both a position or $q$ (namely, $\phi_1I$) and a direction of $p$ (namely, $\phi^\sharp_I(j)$).
\end{example}

The diagram in \eqref{eqn.closure} used to construct the polynomial $[p,q]$ also provides a straightforward proof that $[-,-]$ is a closure for the monoidal structure $(\yon,\otimes)$.

\begin{theorem}\label{closure}
For $\E$ a cartesian closed finite limit category, the monoidal category $(\poly_\E,\yon,\otimes)$ has a closure given by $[-,-]$.
\end{theorem}

\begin{proof}
To show that the functors $p \otimes -$ and $[p,-]$ are adjoints, it suffices to define the unit and counit transformations which take the form of morphisms of polynomials $q \to [p,p \otimes q]$ (``pairing'') and $p \otimes [p,q] \to q$ (``evaluation'') respectively.

The evaluation map is a diagram of the form in \eqref{eqn.evaluation},
\begin{equation}\label{eqn.evaluation}
\begin{tikzcd}
\left(Q_\ast \times_Q \left(q^p \times P\right)\right)\times P_\ast \rar & q^p\times P \dar[equals] \\
\bullet \uar[dashed] \rar[dashed] \dar[dashed] & q^p\times P \dar[dashed] \\
Q_\ast \ar[r, "q"'] & Q
\end{tikzcd}
\end{equation}
where the map $q^p\times P \to Q$ is given by the composite of downward-pointing maps of that form in \eqref{eqn.closure}. The diagram in \eqref{eqn.closure} shows that the pullback of this map along $q$ is precisely $Q_\ast \times_Q \left(q^p \times P\right)$, so the desired map $\bullet \to \left(Q_\ast \times_Q \left(q^p \times P\right)\right)\times P_\ast$ is induced by the identity and the composite map $Q_\ast \times_Q \left(q^p \times P\right) \to P_\ast$ on the left side of the diagram in \eqref{eqn.closure}, which commutes over $q^p\times P$ by inspection of the left side of the diagram in \eqref{eqn.closure}.

The pairing map takes the form of the dotted arrows in \eqref{eqn.pairing}, where the diagram below them is that used in the construction of $[p,p \otimes q]$ from \eqref{eqn.closure}.
\begin{equation}\label{eqn.pairing}
\begin{tikzcd}[row sep=scriptsize]
&[-55pt] &[-5pt] Q_\ast \ar{rrrr} &[-20pt] &[-25pt] &[-50pt] &[0pt] Q \dar[equals] \\
& & \bullet_5 \dar[dashed] \ar[dashed]{rrrr} \uar[dashed] & {} & & & Q \dar[dashed] \\
& & \bullet_4 \ar{ddl} \ar{rrr} & {} & & (p \otimes q)^p \times P \ar{dl} \rar & (p \otimes q)^p \ar{dd} \\
& & & {} & \bullet_3 \ar{dl} \ar{dr} & & {} \\
& \bullet_2 \ar{dl} \ar{rr} & & \bullet_1 \ar{ddr} & {} & (P \times Q)^P \times P \ar{ddl} \rar & (P \times Q)^P \ar{ddd} \\
P_\ast \times Q_\ast \times P_\ast \ar{dr} & & {} & & & & {}\\
& P_\ast \times Q_\ast \times P \ar{rrr} & & & P \times Q \times P \ar{dr} \\
& & & & & P \rar & 1
\arrow[phantom, from=2-3, to=3-4, pos=.05, "\lrcorner"]
\arrow[phantom, from=3-3, to=4-4, pos=.05, "\lrcorner"]
\arrow[phantom, from=3-6, to=4-7, pos=.05, "\distpb"]
\arrow[phantom, from=4-5, to=5-5, near start, "\vee"]
\arrow[phantom, from=5-2, to=6-3, pos=.05, "\distpb"]
\arrow[phantom, from=5-6, to=6-7, pos=.05, "\distpb"]
\end{tikzcd}
\end{equation}
Using the universal property of distributivity squares, a map $Q \to (p \otimes q)^p$ is uniquely determined by a map $Q \to (P \times Q)^P$ whose pullback over $(P \times Q)^P \times P$ factors through $\bullet_3$. The desired map is precisely the pairing map for the cartesian closed structure on the category $\E$, and it pulls back to $Q \times P$. By the universal property of pullbacks, a map $Q \times P \to \bullet_3$ commuting over $(P \times Q)^P \times P$ is determined by a map $Q \times P \to \bullet_1$ commuting over $P \times Q \times P$, where the composite $Q \times P \to (P \times Q)^P \times P \to P \times Q \times P$ can be straightforwardly checked to be given by the diagonal map on $P$. By the universal property of distributivity squares, a map $Q \times P \to \bullet_1$ commuting over $P \times Q \times P$ amounts to a map from the pullback of $Q \times P$ over $P_\ast \times Q_\ast \times P$ to $P_\ast \times Q_\ast \times P_\ast$.

It is a general fact that in a category with finite limits and morphisms $f \colon A \to C$ and $g \colon B \to C$, the pullback of the diagonal map $C \to C \times C$ along $f \times g$ is the map $A \times_C B \to A \times B$, so we have 
\[
(Q \times P) \times_{P \times Q \times P} P_\ast \times Q_\ast \times P \cong Q_\ast \times P_\ast \times_P P \cong Q_\ast \times P_\ast.
\]
Therefore, the desired map is of the form $Q_\ast \times P_\ast \to P_\ast \times Q_\ast \times P_\ast$, and hence the diagonal map for $P_\ast$ is satisfactory as it clearly commutes over $P_\ast \times Q_\ast$ with the map $P_\ast \times_P P \to P_\ast \times P$.

In the process of constructing the map $Q \to (p \otimes q)^p$, we have shown using the factorization property of pullbacks that $\bullet_5 \cong Q_\ast \times P_\ast$ as $Q$ pulls back along the composite map $[p,p \otimes q]$ first to $Q \times P$ and then to $Q_\ast \times P_\ast$. Therefore the map $\bullet_5 \to Q_\ast$ can be chosen to be the canonical projection, which commutes over $Q$ as the dotted horizontal arrow $\bullet_5 \to Q$ factors through the projection $Q \times P \to Q$.

While we do not directly prove the triangle equations here to avoid diagrams too large to easily navigate in this format, the constructions of the evaluation and pairing maps directly from the form of the diagram in \eqref{eqn.closure} makes these equations straightforward (albeit tedious) to check. 
\end{proof}

Combining \cref{closure} with the duoidal structure on $\poly_\E$ from \cref{duoidal} lets us describe a useful relationship between the closure $[-,-]$ and the composition product $\tri$.

\begin{corollary}
For $\E$ a cartesian closed finite limit category and $q$ a polynomial in $\E$, $[\yon,q] \cong q$, 
and for polynomials $p_1,p_2,q_1,q_2$ in $\E$ there is a morphism of polynomials of the form
\[
[p_1,q_1] \tri [p_2,q_2] \to [p_1 \tri p_2,q_1 \tri q_2].
\]
\end{corollary}

\begin{proof}
The first claim follows immediately from the definition, as when $p$ is given by $\yon$ the rightmost morphism in the top row of \eqref{eqn.closure} is the identity on $Q$.

For the second claim, this morphism arises from applying the following sequence of evaluation, interchange \eqref{interchange}, and transpose maps to the identity morphism on $q_1 \tri q_2$.
\[
\begin{array}{rcl}
\Hom_{\poly_\E}\left(q_1 \tri q_2, q_1 \tri q_2\right)&\to&
\Hom_{\poly_\E}\left((p_1 \otimes [p_1,q_1]) \tri (p_2 \otimes [p_2,q_2]), q_1 \tri q_2\right)\\&\to&
\Hom_{\poly_\E}\left((p_1 \tri p_2) \otimes ([p_1,q_1] \tri [p_2,q_2]), q_1 \tri q_2\right)\\&\cong&
\Hom_{\poly_\E}\left([p_1,q_1] \tri [p_2,q_2], [p_1 \tri p_2,q_1 \tri q_2]\right) 
\end{array}
\]
\end{proof}

\section{Coclosures for $\tri$}\label{sec.coclosures}

Whereas a closure for a monoidal structure $\otimes$ on a category is a natural right adjoint to the functors $a \otimes -$, a \emph{coclosure} is a natural left adjoint. The monoidal structure on the category $\poly_\E$ given by $\tri$ is not symmetric, so there is furthermore a distinction between a left- and right-coclosure for $\tri$, which are left adjoints to functors of the form $p \trii -$ and $- \trii p$ respectively. We show that $\tri$ has a partially-defined right-coclosure, which is total when $\E$ is locally cartesian closed, and an \emph{indexed} left-coclosure for any finite limit category.

\begin{definition}
For polynomials $p$ and $q$ in a finite limit category $\E$, the morphism $\lens{p}{q}$ in $\E$ is given by the composite of vertical arrows in \eqref{eqn.rightcoclosure},
\begin{equation}\label{eqn.rightcoclosure}
\begin{tikzcd}
&[-30pt] \bullet \ar{dl} \rar \ar[phantom]{dr}[pos=.1]{\distpb} & \Pi_{\id_P \times q} \left(p \times \id_{Q_\ast}\right) \ar{dd} \\
P_\ast \times Q_\ast \ar{dr} & & {} \\
& P \times Q_\ast \rar & P \times Q \dar \\
& & P
\end{tikzcd}
\end{equation}
where the morphism $P \times Q_\ast \To{\id_P \times q} P \times Q$ is exponentiable by \cref{productpi} and the pentagon is a distributivity pullback. 
\end{definition}

Note that $\lens{p}{q}$ is not generally exponentiable. While the projection $P \times Q \to P$ is exponentiable when $\E$ is cartesian closed by \cref{productpi}, the functor $\Pi_{\id_P \times q}$ does not generally preserve exponentiable morphisms. Only when $\E$ is locally cartesian closed, and hence all morphisms are exponentiable, can $\lens{p}{q}$ be assumed to be a polynomial without any assumptions on $p,q$.

\begin{example}
In $\poly_{\smset}$, we have from \cite[Proposition 2.1.15]{spivak2021functorial} the formula for the right coclosure
\[
\lens{p}{q} = \sum_{I \in P} \yon^{\sum_{J \in Q} p[I]^{q[J]}},
\]
corresponding to the dependent projection function $\sum_{I,J} p[I]^{q[J]} \to P$. Here it is straightforward to see that a morphism $\lens{p}{q} \to r$ consists of a function $\phi_1 \colon P \to R$, for each $I \in P$ a function $\phi^\sharp_I \colon r[\phi_1I] \to Q$, and for each $k \in r[\phi_1I]$ a function $q[\phi^\sharp_Ij] \to p[I]$, precisely the data of a morphism $p \to r \tri q$.
\end{example}

\begin{example}\label{coclosurecounterexample}
In $\poly_\smcat$ there is a similar formula, where the total space has objects $\sum_{I,J} p[I]^{q[J]}$ for $I \in P$ and $J \in Q$, and a morphism from $h \colon q[J] \to p[I]$ to $h' \colon q[J'] \to p[I']$ consists of morphisms $f \colon I \to I'$ in $P$, $g \colon J \to J'$ in $Q$, and a square in $\prof$ of the form
\[
\begin{tikzcd}
{q[J]} \rar[tick,""{name=U, below}]{{q[g]}} \dar[swap]{h} & {q[J']} \dar{h'} \\
{p[I]} \rar[tick,swap,""{name=V, above}]{{p[f]}} & {p[I']}.
\arrow[Rightarrow,shorten >=8,shorten <=5,from=U,to=V-|U]
\end{tikzcd}
\]

For an example of when $\lens{p}{q}$ is not exponentiable, let $P$ and $Q$ both be the ordinal category $0 \to 1 \to 2$, set $q[-] \colon Q \to \prof$ to send each object to the terminal category and each non-identity morphism to the empty profunctor, while $p[-] \colon P \to \prof$ sends $0$ and $2$ to the terminal category, $1$ to the empty category, and all non-identity morphisms to empty profunctors. In $\lens{p}{q}$, there is a morphism in the total space as on the left in \eqref{eqn.noclosure} sent to $0 \to 2$ in $P$,
\begin{equation}\label{eqn.noclosure}
\begin{tikzcd}
1 \rar[tick,""{name=U, below}]{\varnothing} \dar[equals] & 1 \dar[equals] \\
1 \rar[tick,swap,""{name=V, above}]{\varnothing} & {1}
\arrow[equals,shorten=4,from=U,to=V]
\end{tikzcd}\qquad\qquad\begin{tikzcd}
1 \rar[tick,""{name=S, below}]{\varnothing} \dar[equals] & 1 \rar[tick,""{name=U, below}]{\varnothing} \dar[dashed] & 1 \dar[equals] \\
1 \rar[tick,swap,""{name=T, above}]{\varnothing} & \varnothing \rar[tick,swap,""{name=V, above}]{\varnothing} & 1
\arrow[Rightarrow,dashed,shorten <=4,shorten >=7,from=S,to=T-|S]
\arrow[Rightarrow,dashed,shorten <=4,shorten >=7,from=U,to=V-|U]
\end{tikzcd}
\end{equation}
but there is no lift of the factorization $0 \to 1 \to 2$ in $P$ to a factorization of this morphism in the total space, which would have the form as on the right in \eqref{eqn.noclosure}. This is because there is no functor from the terminal category into the empty category.
\end{example}

We now show that when it exists as a polynomial, $\lens{-}{-}$ is a right-coclosure for $\tri$ in the category $\poly_\E$. Note that it is straightforward to check that among choices of $p,q$ where $\lens{p}{q}$ is a polynomial, it is covariantly functorial in $p$ and contravariantly functorial in $q$.

\begin{theorem}\label{rightcoclosure}
For polynomials $p,q,r$ in a category $\E$ with finite limits, if $\lens{p}{q}$ is a polynomial then we have
\[
\Hom_{\poly_\E}\left(p,r \trii q\right) \cong \Hom_{\poly_\E}\left(\lens{p}{q},r\right)
\]
naturally in suitable choices of $p,q,r$.
\end{theorem}

\begin{proof}
A morphism $\phi$ of polynomials from $\lens{p}{q}$ to $r$ consists of the dotted arrows in \eqref{eqn.fromrightcoclosure},
\begin{equation}\label{eqn.fromrightcoclosure}
\begin{tikzcd}
&[-30pt] & P \times_R R_\ast \dar[dashed,swap]{\phi^\sharp} \rar[dashed] \ar[phantom]{dr}[pos=.1]{\lrcorner} & R_\ast \ar{dddd} \\
& \bullet \ar{dl} \rar \ar[phantom]{dr}[pos=.1]{\distpb} & \Pi_{\id_P \times q} \left(p \times \id_{Q_\ast}\right) \ar{dd} & {} \\
P_\ast \times Q_\ast \ar{dr} & & {} \\
& P \times Q_\ast \rar{\id_P \times q} & P \times Q \dar \\
& & P \rar[dashed]{\phi_1} & R
\end{tikzcd}
\end{equation}
while a morphism $\psi$ of polynomials from $p$ to $r \trii q$ consists of the dotted arrows in \eqref{eqn.totri}.
\begin{equation}\label{eqn.totri}
\begin{tikzcd}
P_\ast \ar{rrrr} & &[-20pt] &[-20pt] & P \dar[equals] \\
Q_\ast \times_Q (P \times_R R_\ast) \uar[dashed]{\psi^\sharp} \dar[dashed] \ar[dashed]{rr} \ar[phantom]{ddrr}[pos=0]{\lrcorner} & & P \times_R R_\ast \dar[dashed] \ar[dashed]{rr} \ar[phantom]{ddrr}[pos=0]{\lrcorner} & & P \dar[dashed]{\psi_1} \\
\bullet \ar{dd} \ar{rr} \ar[phantom]{ddrr}[pos=0]{\lrcorner} & & \bullet \ar{rr} \dar \ar[phantom]{ddrr}[pos=0]{\distpb} & & \bullet \ar{dd} \\
& & Q \times R_\ast \ar{dl} \ar{dr} & & {} \\
Q_\ast \rar & Q & {} & R_\ast \rar & R 
\end{tikzcd}
\end{equation}

Given a morphism of polynomials $\phi \colon \lens{p}{q} \to r$ as in \eqref{eqn.fromrightcoclosure}, we can construct a morphism $\psi \colon p \to r \trii q$ as follows. By the universal properties of distributivity squares and products, the desired map $\psi_1$ is uniquely determined by a map $P \to R$ and a map from $P \times_R R_\ast$ to $Q$; the former is given by $\phi_1$ and the latter by the map $P \times_R R_\ast \to P \times Q$ in \eqref{eqn.fromrightcoclosure} composed with the projection to $Q$. As the map $P \times_R R_\ast \to Q$ in \eqref{eqn.totri} factors through the projection $P \times Q \to Q$ and $\id_P \times q$ is apullback of $q$, the pullback $Q_\ast \times_Q (P \times_R R_\ast)$ in \eqref{eqn.totri} agrees with the pullback of $P \times_R R_\ast$ along $\id_P \times q$ in \eqref{eqn.fromrightcoclosure} (not pictured). The projection from this pullback to $P \times Q_\ast$ factors through $P_\ast \times Q_\ast$ and therefore by projection maps to $P_\ast$ in a manner commuting over $P$, providing the definition of $\psi^\sharp$ needed to complete the definition of $\psi$.

For the converse, starting with $\psi$ as in \eqref{eqn.totri} and constructing a morphism of polynomials $\phi$ as in \eqref{eqn.fromrightcoclosure}, $\phi_1$ can be defined as the composite vertical map in \eqref{eqn.totri} from $P$ to $R$. Its pullback along $r$ maps to both $P$ and $Q$ and hence to $P \times Q$ as in \eqref{eqn.fromrightcoclosure}. To define the map $\phi^\sharp$ requires such a morphism $P \times_R R_\ast$ along with a map from its pullback along $\id_P \times q$ to $P_\ast \times Q_\ast$. As discussed above this pullback is precisely the pullback $Q_\ast \times_Q (P \times_R R_\ast)$ depicted in \eqref{eqn.totri}, which maps to both $P_\ast$ (commuting over $P$) and $Q_\ast$, and hence to their product.

Naturality of these constructions is straightforward to check as they rely only on universal properties, but we omit this for brevity.
\end{proof}

While the mapping property holds for $\lens{-}{-}$ merely as a morphism in the category $\E$ (since the proof of \cref{rightcoclosure} makes no reference to its exponentiability beyond the fact that morphisms of this sort were only defined for polynomials), when $\E$ is locally cartesian closed a more formal statement can be made.

\begin{corollary}
For a locally cartesian closed category $\E$ with all finite limits, the operation $\lens{-}{-}$ is a right-coclosure for the monoidal category $(\poly_\E,\yon,\tri)$.
\end{corollary}

There are numerous properties of the right-coclosure that could be proven in this setting, but in the interest of space we provide an example of just one describing its interaction with the Dirichlet monoidal structure on $\poly_\E$.

\begin{corollary}
For $\E$ a finite limit category and $p$ is a polynomial in $\E$ we have $\lens{p}{\yon} \cong p$, and for $p_1,p_2,q_1,q_2$ polynomials in $\E$ we have a natural morphism of polynomials
\[
\lens{p_1 \otimes p_2}{q_1 \otimes q_2} \to \lens{p_1}{q_1} \otimes \lens{p_2}{q_2}
\]
whenever all of the coclosures involved are polynomials.
\end{corollary}

\begin{proof}
That $\lens{-}{\yon}$ is isomorphic to the identity functor follows immediately from the definition as $\Pi_{\id}$ is always an equivalence of categories.

The remaining natural morphism is constructed from the duoidal relationship between $\otimes$ and $\tri$, by applying the map below to the identity morphisms on $\lens{p_1}{q_1}$ and $\lens{p_2}{q_2}$.
\begin{align*}
\Hom_{\poly_\E}\left(\lens{p_1}{q_1}, \lens{p_1}{q_1}\right) &\times \Hom_{\poly_\E}\left(\lens{p_2}{q_2}, \lens{p_2}{q_2}\right)\\&\qquad\cong 
\Hom_{\poly_\E}\left(p_1, \lens{p_1}{q_1} \trii q_1\right) \times \Hom_{\poly_\E}\left(p_2, \lens{p_2}{q_2} \trii q_2\right)\\&\qquad\to 
\Hom_{\poly_\E}\left(p_1 \otimes p_2, \left(\lens{p_1}{q_1} \trii q_1\right) \otimes \left(\lens{p_2}{q_2} \trii q_2\right) \right)\\&\qquad\to 
\Hom_{\poly_\E}\left(p_1 \otimes p_2, \left(\lens{p_1}{q_1} \otimes \lens{p_2}{q_2}\right) \tri \left(q_1 \otimes q_2\right) \right)\\&\qquad\cong 
\Hom_{\poly_\E}\left(\lens{p_1 \otimes p_2}{q_1 \otimes q_2}, \lens{p_1}{q_1} \otimes \lens{p_2}{q_2}\right)
\end{align*}
\end{proof}

The monoidal structure $(\yon,\tri)$ on $\poly_\E$ also has something similar to a left-coclosure, but unlike the mapping property in \cref{rightcoclosure} here the set of maps from a polynomial $p$ to a composite $q \trii r$ is isomorphic to a disjoint union of many different sets of morphisms into $r$; this is called an indexed left-coclosure for $(\yon,\trii)$.

\begin{definition}
For polynomials $p,q$ in a finite limit category $\E$ and a morphism $P \To{f} Q$ in $\E$, the polynomial $p \frowny{f} q$ is defined by the diagram of pullbacks in \eqref{eqn.frown}.
\begin{equation}\label{eqn.frown}
\begin{tikzcd}
P_\ast \times_Q Q_\ast \rar{p \frowny{f} q} \dar \ar[phantom]{dr}[pos=.05]{\lrcorner} & P \times_Q Q_\ast \dar \rar \ar[phantom]{dr}[pos=.05]{\lrcorner} & Q_\ast \dar{q} \\
P_\ast \rar{p} & P \rar{f} & Q
\end{tikzcd}
\end{equation}
\end{definition}

\begin{example}
In $\poly_{\smset}$, for a function $P \to Q$ we have the formula
\[
p \frowny{f} q \cong \sum_{I \in P} q[fI] \times \yon^{p[I]}.\qedhere
\]
\end{example}

Note that for a cartesian morphism $\phi \colon q \to q'$ there is an isomorphism 
\[
p\frowny{f}q \cong p\frowny{\phi_1 \circ f}q'.
\]
For a morphism of polynomials $\gamma \colon p' \to p$, there is a morphism 
\[
\gamma \frowny q \colon p' \frowny{f \circ \gamma_1} q \to p \frowny{f} q
\]
given on positions by $\gamma_1 \times_Q \id_{Q_\ast} \colon P' \times_Q Q_\ast \to P \times_Q Q_\ast$ and on directions by 
\[
(P' \times_Q Q_\ast) \times_{(P \times_Q Q_\ast)} (P_\ast \times_Q Q_\ast) \cong (P' \times_P P_\ast) \times_Q Q_\ast \To{\gamma^\# \times_Q \id_{Q_\ast}} P'_\ast \times_Q Q_\ast.
\]

One canonical example of this construction is in the case when the morphism $f$ is an identity, from which we recover some familiar notation.

\begin{example}\label{pstar}
The polynomial $p \frowny{\id_P} p$ is the pullback of $p$ along itself, resulting in a morphism $P_\ast \times_P P_\ast \to P_\ast$. Based on this form, we denote the morphism $p \frowny{\id_P} p$ as
\[P_{\ast\ast} \To{p_\ast} P_\ast,\] 
making the existing notation of ``$P_\ast$'' for the domain of a polynomial $p$ with codomain $P$ an instance of this more general construction of a polynomial $p_\ast$ whose codomain is $P_\ast$. In the case of $\poly_{\smset}$, we have 
\[
p_\ast = \sum_{I \in P} p[I]\yon^{p[I]}.
\]

There is by definition a cartesian morphism of polynomials $p_\ast \to p$, and by the construction of pullbacks of cartesian morphisms in \cref{hasconnectedlimits} along with the cancellation property of pullback squares, if $p \to q$ is a cartesian morphism of polynomials then its pullback along $q_\ast \to q$ is precisely $p_\ast$.
\end{example}

As the definition of $p \frowny{f} q$ uses only pullbacks in the category $\E$, commutation of products with limits immediately shows the following relationship with the tensor product.

\begin{lemma}
For polynomials $p_1,p_2,q_1,q_2$ in a finite limit category $\E$, and morphisms $P_1 \To{f_1} Q_1$ and $P_2 \To{f_2} Q_2$ in $\E$, we have a natural isomorphism
\[
p_1 \otimes p_2 \frowny{f_1 \times f_2} q_1 \otimes q_2 \cong \left(p_1 \frowny{f_1} q_1\right) \otimes \left(p_2 \frowny{f_2} q_2\right).
\]
\end{lemma}

Finally, we show that $\frown$ is an indexed left-coclosure.

\begin{theorem}\label{indexedleftcoclosure}
The operation $\frown$ is an indexed left-coclosure for $\tri$, in the sense that for polynomials $p,q,r$ in a finite limit category $\E$,
\[
\Hom_{\poly_\E}\left(p, q \trii r\right) \cong \coprod_{f \colon P \to Q} \Hom_{\poly_\E}\left(p \frowny{f} q, r\right)
\]
naturally over all maps in $p$ and $r$ and cartesian maps in $q$.
\end{theorem}

\begin{proof}
A morphism of polynomials $\psi \colon p \to q \trii r$ consists of the dashed edges in \eqref{eqn.lefttotri},
\begin{equation}\label{eqn.lefttotri}
\begin{tikzcd}
P_\ast \ar{rrrr} & &[-20pt] &[-20pt] & P \dar[equals] \\
R_\ast \times_R (P \times_Q Q_\ast) \uar[dashed]{\psi^\sharp} \dar[dashed] \ar[dashed]{rr} \ar[phantom]{ddrr}[pos=0]{\lrcorner} & & P \times_Q Q_\ast \dar[dashed] \ar[dashed]{rr} \ar[phantom]{ddrr}[pos=0]{\lrcorner} & & P \dar[dashed]{\psi_1} \\
\bullet \ar{dd} \ar{rr} \ar[phantom]{ddrr}[pos=0]{\lrcorner} & & \bullet \ar{rr} \dar \ar[phantom]{ddrr}[pos=0]{\distpb} & & \Pi_q (R \times Q_\ast) \ar{dd} \\
& & R \times Q_\ast \ar{dl} \ar{dr} & & {} \\
R_\ast \rar & R & {} & Q_\ast \rar & Q 
\end{tikzcd}
\end{equation}
where we denote by $f$ the composite morphism $P \to Q$. A morphism of polynomials $\phi \colon p \frowny{f} q \to r$ consists of morphisms $\phi_1 \colon P \times_Q Q_\ast \to R$ and 
\[
\phi^\sharp \colon R_\ast \times_R (P \times_Q Q_\ast) \to P_\ast \times_Q Q_\ast \cong P_\ast \times_P (P \times_Q Q_\ast).
\]

Given $\psi$ as in \eqref{eqn.lefttotri}, define $\phi_1$ as the vertical composite map $P \times_Q Q_\ast\to R$ in the center of \eqref{eqn.lefttotri}. As $R_\ast \times_R (P \times_Q Q_\ast)$ maps to both $P_\ast$ and $P \times_Q Q_\ast$ commuting over $P$, it has an induced map $\phi^\sharp$ to $P_\ast \times Q_\ast$.

This is natural with respect to cartesian maps out of $q$ as by the composition of pullback squares, given a cartesian morphism $q \to q'$ we have that $P \times_Q Q_\ast \cong P \times_{Q'} Q'_\ast$. Given a morphism of polynomials $\gamma \colon p' \to p$, applying this construction to the composite $\psi \circ \gamma$ results in a morphism given by the maps 
\[
P' \times_Q Q_\ast \To{\gamma_1 \times_Q \id_{Q_\ast}} P \times_Q Q_\ast \To{\phi_1} R
\]
on positions by composition of pullback squares, and similarly
\[
R_\ast \times_R (P' \times_Q Q_\ast) \cong P' \times_P (R_\ast \times_R (P \times_Q Q_\ast)) \To{\id_{P'} \times_P \phi^\#} P' \times_P (P_\ast \times Q_\ast) \cong (P' \times_P P_\ast) \times Q_\ast \To{\gamma^\# \times \id_{Q_\ast}} P'_\ast \times Q_\ast
\]
on directions, which agrees definitionally with the composite morphism $p' \frowny{f \circ \gamma_1} q \to p \frowny{f} q \To{\phi} r$. Naturality in $r$ similarly follows from the definitions of composition of morphisms of polynomials and functoriality of $\tri$.

For the inverse construction, given $\phi \colon p \frowny{f} q \to r$, the map $\phi_1$ determines a map $P \times_Q Q_\ast \to R \times Q_\ast$, where $P\times_Q Q_\ast$ is the pullback of $q$ along $f$. This induces a map $\psi_1 \colon P \to \Pi_q (R \times Q_\ast)$ by the universal properties of products and distributivity pullbacks. The map $\phi^\sharp$ then projects to a morphism $\psi^\sharp \colon R_\ast \times_R (P \times_Q Q_\ast) \to P_\ast$ commuting over $P$. This construction of $\psi$ from $\phi$ is easily seen to invert that of $\phi$ from $\psi$ given above, from which the relevant naturality conditions follow automatically.
\end{proof}

We can now exploit the construction of the operation $\frown$ and general properties of monoidal products with indexed left-coclosures to show that $\tri$ preserves connected limits in the second variable.

\begin{corollary}\label{trilimit}
For $q$ a polynomial in a finite limit category $\E$, the functor $q \tri - \colon \poly_\E \to \poly_\E$ preserves connected limits and the functor $- \tri q \colon \poly_\E \to \poly_\E$ preserves connected limits of diagrams whose morphisms are cartesian.
\end{corollary}

\begin{proof}
The first claim is in fact the case for any functor with an indexed left adjoint as coproducts commute with connected limits in the category of sets. This is because by the Yoneda lemma, for any connected diagram $r_{(-)}$ in $\poly_\E$ it suffices to observe the following composite of isomorphisms natural in the polynomial $p$.
\begin{multline*}
\Hom_{\poly_\E}(p,q \tri \lim_i r_i) 
\cong \coprod_{f \colon P \to Q} \Hom_{\poly_\E}(p \frowny{f} q, \lim_i r_i) 
\cong \coprod_{f \colon P \to Q} \lim_i \Hom_{\poly_\E}(p \frowny{f} q, r_i)
\\
\cong \lim_i \coprod_{f \colon P \to Q} \Hom_{\poly_\E}(p \frowny{f} q, r_i) 
\cong \lim_i \Hom_{\poly_\E}(p,q \tri r_i) 
\end{multline*}

For the second claim, first observe that for a connected diagram $r_{(-)}$ whose morphisms are cartesian we have
\[
\Hom_{\poly_\E}(p,(\lim_i r_i) \tri q) 
\cong \coprod_{f \colon P \to \lim_i R_i} \Hom_{\poly_\E}(p \frowny{f} \lim_i r_i, q)
\]
natural in $p$. As the construction of $p \frowny{f} \lim_i r_i$ is invariant under the cartesian morphisms $\lim_i r_i \to r_{i_0}$ for each choice of $i_0$, we have 
\[
p \frowny{f} \lim_i r_i \cong p \frowny{(\pi_{i_0})_1 \circ f} r_{i_0}.
\]
Based on this, and the distributivity of limits over coproducts in the category of sets, $\Hom_{\poly_\E}(p,(\lim_i r_i) \tri q)$ is furthermore isomorphic to 
\[
\lim_i \coprod_{f \colon P \to R_i} \Hom_{\poly_\E}(p \frowny{(\pi_i)_1 \circ f} r_i, q) 
\cong \lim_i \Hom_{\poly_\E}(p,r_i \tri q)
\cong \Hom_{\poly_\E}(p,\lim_i(r_i \tri q)),
\]
completing the proof.
\end{proof}

\begin{remark}
The fact that $\tri$ preserves connected limits of cartesian morphisms could be proven directly as the cartesian projection morphisms fit easily into the composition diagrams defining the composition product $\tri$ in \cref{tri}, but using the left coclosure provides a convenient syntax for expressing the same ideas. That $\tri$ preserves arbitrary connected limits in the left variable would be much more tedious to prove without the indexed left-coclosure formalism.
\end{remark}

\chapter{Comonoids and Bicomodules in $\poly_E$}\label{chap.comonoid}

As in any monoidal category, one can ask what mathematical structures are described by its monoids and comonoids. While we will not discuss monoids in $\poly_\E$ at this time, $\tri$-comonoids in $\poly_\E$ correspond to certain internal categories in $\E$ with their homomorphisms corresponding to internal cofunctors. Meanwhile, coalgebras for a comonoid when regarded as an endofunctor on $\E$ correspond to internal copresheaves (a.k.a. internal discrete opfibrations) and bicomodules between comonoids allow us to reconstruct the bicategory of typed polynomials, hence generalizing the results of \cite{kock2012polynomial} and \cite{weber2015polynomials}.

\section{Comonoids are internal categories}

While in classical mathematics comonoids receive far less attention than monoids, in the monoidal category $(\poly_\E,\yon,\tri)$ they recover an unexpected category: that of internal categories and cofunctors in $\E$. This surprising correspondence was first observed in the case when $\E = \smset$ by Ahman and Uustalu in \cite{ahman2016directed,ahman2017taking}, where they show that comonoids and comonoid homomorphisms in $\poly_{\smset}$ correspond to categories and cofunctors. Previously in \cite[Section 7]{ahman2014when}, ``directed containers'' were defined in a category $\E$ with pullbacks, which can be observed to agree definitionally with categories internal to $\E$ (see \cref{def.internal_cat}) whose source morphism is exponentiable, and it was suggested that these directed containers can be interpreted as polynomial comonads. We now complete this story by proving that when $\E$ has all finite limits (so that polynomials induce endofunctors on $\E$ itself rather than only its slice categories), $\tri$-comonoids and comonoid homomorphisms in $\poly_\E$ correspond to internal categories with exponentiable source and internal cofunctors between them.

\begin{definition}\label{comonoid}
A $\tri$-comonoid in $\poly_\E$ is a polynomial $c$ equipped with maps $\epsilon \colon c \to \yon$ (the \emph{counit}) and $\delta \colon c \to c \trii c$ (the \emph{comultiplication}) such that the counit and coassociativity diagrams in \eqref{eqn.comon} commute.
\begin{equation}\label{eqn.comon}
\begin{tikzcd}[column sep=large]
\yon \trii c \ar{dr}[swap]{\cong} & c \trii c \lar[swap]{\epsilon \trii \id_c} \rar{\id_c \trii \epsilon} & c \trii \yon \ar{dl}{\cong} \\
& c \uar{\delta}
\end{tikzcd}\qquad\quad\begin{tikzcd}[row sep=tiny, column sep=large]
(c \trii c) \trii c \ar{dd}[swap]{\cong} & c \trii c \lar[swap]{\delta \trii \id_c} \\
& &[-10pt] c \ar{ul}[swap]{\delta} \ar{dl}{\delta} \\
c \trii (c \trii c) & c \trii c \lar{\id_c \trii \delta}
\end{tikzcd}
\end{equation}
A comonoid homomorphism $c \to c'$ is a morphism of the underlying polynomials which commutes with the respective counit and comultiplication maps. We write $\comon(\poly_\E)$ for the category of comonoids and homomorphisms.
\end{definition}

\begin{example}
For any polynomial $p$, the polynomial $p_\ast$ from \cref{pstar} has a canonical comonoid structure, where the counit is obtained by transposing the isomorphism $p \to p \tri \yon$ (using the indexed left coclosure) into a morphism $p \frowny{\id} p \to \yon$. The comultiplication 
\[
p \frowny{\id} p\too \left(p \frowny{\id} p\right) \tri \left(p \frowny{\id} p\right)
\]
is the transpose of the composite morphism 
\[
p \to p \tri \left(p \frowny{\id} p\right) \to p \tri \left(p \frowny{\id} p\right) \tri \left(p \frowny{\id} p\right)
\]
given by twice applying the transpose of the identity on $p \frowny{\id} p$ to a map $p \to p \tri (p \frowny{\id} p)$.
\end{example}

\begin{example}
For a polynomial $p$ such that $\lens{p}{p}$ is exponentiable, it too carries a canonical comonoid structure, where the counit is obtained by transposing the isomorphism $p \to \yon \tri p$ (using the right coclosure) into a morphism $\lens{p}{p} \to \yon$. The comultiplication $\lens{p}{p} \to \lens{p}{p} \tri \lens{p}{p}$ is the transpose of the morphism
\[
p \to \lens{p}{p} \tri p \to \lens{p}{p} \tri \lens{p}{p} \tri p
\]
given by twice applying the transpose of the identity on $\lens{p}{p}$ to a map $p \to \lens{p}{p} \tri p$.
\end{example}

While \cref{comonoid} describes a structure in the category $\E$ rather opaquely, we will show that in fact it is the same as that of an internal category in $\E$ with one additional condition.

\begin{definition}\label{def.internal_cat}
A \emph{category $\C$ internal to $\E$} consists of a diagram in $\E$ of the form 
\[
\begin{tikzcd}
C_1 \rar[shift left=2]{s} \rar[shift right=2,swap]{t} & C_0
\end{tikzcd}
\]
equipped with an \emph{identity} map $i \colon C_0 \to C_1$ satisfying $s \circ i = \id_{C_0} = t \circ i$ and a \emph{composition} map $k \colon C_1 \pb{t}{s} C_1 \to C_1$ satisfying $s \circ k = s \circ \pi_1$ and $t \circ k = t \circ \pi_2$, such that furthermore the unit and associativity diagrams in \eqref{eqn.internalcat} commute.
\begin{equation}\label{eqn.internalcat}
\begin{tikzcd}[column sep=huge]
C_0 \pb{\id}{s} C_1 \ar{dr}[swap]{\cong} \rar{i \times_{C_0} \id_{C_1}} & C_1 \pb{t}{\id} C_1 \dar{k} & C_1 \;\pb{t}{s} C_0 \ar{dl}{\cong} \lar[swap]{\id_{C_1} \times_{C_0} i} \\
& C_1 
\end{tikzcd}
\end{equation}
\[
\begin{tikzcd}[row sep=tiny, column sep=huge]
(C_1 \pb{t}{s} C_1) \pb{t \circ \pi_2}{s} C_1 \ar{dd}[swap]{\cong} \rar{k \times_{C_0} \id_{C_1}} & C_1 \pb{t}{s} C_1 \ar{dr}{k} \\
& &[-10pt] C_1 \\
C_1 \;\pb{t}{s \circ \pi_1} (C_1 \;\pb{t}{s} C_1) \rar[swap]{\id_{C_1} \times_{C_0} k} & C_1 \pb{t}{s} C_1 \ar{ur}[swap]{k} 
\end{tikzcd}
\]
\end{definition}

It is possible, albeit tedious and involving diagrams larger than we wish to include here, to prove directly that, in any finite limit category, comonoids in $\poly_\E$ are in bijection with internal categories in $\E$ whose source morphism is exponentiable (i.e.\ a polynomial). Instead, to simplify the proof of this claim we make further use of our technique of reducing results about polynomials in any finite limit category to results about polynomials in presheaf categories.

\begin{lemma}\label{preservereflect}
For a finite limit category $\E$ and fully faithful dense functor $F \colon \A \to \E$, the corresponding singular functor $F^\ast \colon \E \to \Ah$ preserves and reflects both internal categories and $\tri$-comonoids. In other words, for diagrams in $\E$ of the form  
\[
\begin{tikzcd}
C_1 \rar[shift left=2]{s} \rar[shift right=2,swap]{t} & C_0
\end{tikzcd}
\qqand C_\ast \To{c} C
\]
where $c$ is exponentiable, there is a bijection between internal category structures on $(C_1,C_0,s,t)$ in $\E$ and internal category structures on $\left(F^\ast(C_1),F^\ast(C_0),F^\ast(s),F^\ast(t)\right)$, and there is also a bijection between $\tri$-comonoid structures on $c$ in $\poly_\E$ and $\tri$-comonoid structures on $F^\ast(c)$ in $\poly_\Ah$.
\end{lemma}

\begin{proof}
This follows immediately as $F^\ast$ is fully faithful and also preserves and reflects finite limits (by definition and finite completeness of $\E$) and exponentials (by \cref{densepi}).
\end{proof}

Based on this, we can now relate $\tri$-comonoid and internal category structures for any finite limit category using only the proof for presheaf categories, which more closely resembles the proof in the case of $\poly_\smset$ from \cite{ahman2016directed}.

\begin{theorem}\label{comonoidcategory}
For a polynomial $C_\ast \To{c} C$ in a finite limit category $\E$, there is a bijection between $\tri$-comonoid structures on $c$ and categories internal to $\E$ with $c$ as their source.
\end{theorem}

\begin{proof}
By \cref{preservereflect}, it suffices to provide this bijection in the case when $\E$ is a presheaf category $\Ah$, as for any fully faithful dense functor $F \colon \A \to \E$ this bijection in $\Ah$ implies the same for $\E$ (and $F$ can always be taken to be the Yoneda embedding, though smaller dense full subcategories of $\E$ are often available).

We therefore start with a $\tri$-comonoid $c \colon C_\ast \to C$ in the category $\poly_\Ah$ and show that its structure is in 1-to-1 correspondence with an internal category structure with the source map $C_\ast \to C$ given by $c$. In doing so we will make extensive use of the notation established in \cref{presheafcomp}; in particular, for $a$ an object of $\A$ and an element $x \in C_a$ the presheaf $c[x]$ is the pullback of $x \colon \yon(a) \to C$ along $c$.

A counit map $\epsilon \colon c \to \yon$ has the form of \eqref{eqn.counit}.
\begin{equation}\label{eqn.counit}
\begin{tikzcd}
C_\ast \dar[swap]{c} &[-5pt] C \lar[swap]{\epsilon^\sharp} \rar \dar[equals] \ar[phantom]{dr}[pos=0]{\lrcorner} &[-5pt] 1 \dar[equals] \\
C & C \lar[equals] \rar{\epsilon_1} & 1
\end{tikzcd}
\end{equation}
The data of the counit $\epsilon$ consists only of the morphism $\epsilon^\sharp \colon C_\ast \from C$ commuting over $C$, which corresponds to the identity map $i \colon C_\ast \from C$ of an internal category structure satisfying $c \circ i = \id_C$. In particular, this counit/identity morphism amounts to a natural choice of element $i(x) \in c[x]_a$ for all elements $x \in C_a$.

A comultiplication map $\delta \colon c \to c \trii c$ has the form of diagram \eqref{eqn.comult}, which takes place in $\smset$ and is natural in the object $a$ in $\A$, using the computation of $c \trii c$ from \cref{presheafcomp}.
\begin{equation}\label{eqn.comult}
\begin{tikzcd}
(C_\ast)_a \dar[swap]{c} &[-5pt] \coprod\limits_{x \in C_a} \coprod\limits_{y \in c[x]_a} c[\delta_1(x)(y)]_a \lar[swap]{\delta^\sharp} \rar \dar \ar[phantom]{dr}[pos=-.2]{\lrcorner} &[-5pt] \coprod\limits_{x \in C_a} \coprod\limits_{f \colon c[x] \to C} \coprod\limits_{y \in c[x]_a} c[f(y)]_a \dar \\
C_a & C_a \lar[equals] \rar{\delta_1} & \coprod\limits_{x \in C_a} \Hom_\Ah\left(c[x], C\right)
\end{tikzcd}
\end{equation}
Note that by the right counit equation for $\delta,\epsilon$ on positions, the composite
\[
C_a \To{\delta_1} \coprod_{x \in C_a} \Hom_\Ah\left(c[x], C\right) \To{(\id_c \trii \epsilon)_1} \coprod_{x \in C_a} \Hom_\Ah\left(c[x], 1\right) \cong C_a
\]
is the identity, which ensures that all of the functions in the right square of \eqref{eqn.comult} act as the identity on $x \in C_a$. 

It follows that the map $\delta_1$ amounts to a choice of function $\delta_1(x) \colon c[x] \to C$ for all $x$ in $C_a$, which by naturality in $a$ are precisely the data of a morphism $t \colon C_\ast \to C$ in $\Ah$, which we take to be the target map for the corresponding internal category structure. Based on this correspondence, we observe that there is an isomorphism
\[
\coprod_{x \in C_a} \coprod_{y \in c[x]_a} c[\delta_1(x)(y)]_a \cong \left(C_\ast \pb{t}{c} C_\ast\right)_a
\]
commuting over $C_a$ with respect to the map 
\[
c \circ \pi_1 \colon \left(C_\ast \pb{t}{c} C_\ast\right)_a \to C_a.
\]
Therefore, the morphism $\delta^\sharp$ commuting over $C$ corresponds to the composition map $k \colon C_\ast \pb{t}{c} C_\ast \to C_\ast$ satisfying $c \circ k = c \circ \pi_1$.

It now remains only to check that given these equivalent structures the counit and coassociativity equations for $\tri$-comonoids correspond to the target, unit, and associativity equations for internal categories.

On positions, the left counit equation for $\delta,\epsilon$ states that the map
\[
C_a \To{\delta_1} \coprod_{x \in C_a} \Hom_\Ah\left(c[x], C\right) \To{(\epsilon \trii \id_c)_1} \coprod_{*_a \in 1_a} \Hom_\Ah\left(\yon(a), C\right) \cong C_a,
\]
sending an element $x \in C_a$ to $\delta_1(x)(\epsilon^\sharp(x)) \in C_a$ is the identity on $C_a$ (recall that $\yon[*_a] = \yon(a)$ as in \cref{presheafcomp}). In the corresponding internal category, this equation is equivalent to that of $t \circ i = \id_C$.

Using the correspondence of $\delta^\sharp$ to the composition map $k \colon \from C_\ast \pb{t}{c} C_\ast \to C_\ast$ and $\epsilon^\ast$ to the identities map $i \colon C_\ast \from C$, it is then straightforward to check that the maps $\left((\epsilon \trii \id_c) \circ \delta\right)^\ast$ and $\left((\id_c \trii \epsilon) \circ \delta\right)^\sharp$ correspond to the composites
\[
C_\ast \cong C \pb{\id}{c} C_\ast \To{i \times_C \id_{C_\ast}} C_\ast \pb{t}{c} C_\ast \To{k} C_\ast
\qqand 
C_\ast \cong C_\ast \pb{t}{\id} C \To{\id_{C_\ast} \times_C i} C_\ast \pb{t}{c} C_\ast \To{k} C_\ast.
\]
Therefore the left and right unit laws on directions for $\epsilon,\delta$ which assert that these maps are identities are equivalent to the left and right unit laws in the corresponding internal category.

The associativity equation for $\delta$ on positions states that the diagram in \eqref{eqn.presheafcoassoc} commutes for all objects $a$ in $\A$,
\begin{equation}\label{eqn.presheafcoassoc}
\begin{tikzcd}
& \coprod\limits_{x \in C_a} \Hom_\Ah\left(c[x], C\right) \rar{(\delta \trii \id_c)_1} &[20pt] \coprod\limits_{\substack{x \in C_a \\ g \colon c[x] \to C}} \Hom_\Ah\left(c[x] \pb{g}{c} C_\ast, C\right) \ar[shorten >=4,shorten <=-10, shift left=5]{dd}{\alpha} \\
C_a \ar{ur}{\delta_1} \ar[shorten >= 5]{dr}[swap]{\delta_1} \\
& \coprod\limits_{x \in C_a} \Hom_\Ah\left(c[x], C\right) \rar{(\id_c \trii \delta)_1} & \coprod\limits_{x \in C_a} \Hom_\Ah\left(c[x], \Pi_c (C \times C_\ast)\right) 
\end{tikzcd}
\end{equation}
where the map $(\delta \trii \id_c)_1$ sends $(x,f)$ to $(x,\delta_1(x), f \circ \delta^\sharp\restrict{x})$ (where $\delta^\sharp\restrict{x}$ is the appropriate restriction of $\delta^\sharp$) and the map $\alpha^{-1} \circ (\id_c \trii \delta)_1$ sends $(x,f)$ to $(x,f, f \times_C t)$. The associativity of this diagram therefore means that for all $x \in C_a$, 
\[
(x,\delta_1(x),\delta_1(x) \circ \delta^\sharp\restrict{x}) = (x,\delta_1(x), \delta_1(x) \times_C t).
\]
Assembling these equations together for all choices of $a$ and $x$ results in the equation 
\[
t \circ \delta^\sharp = t \times_C t,
\]
as maps $C_\ast \pb{t}{c} C_\ast \to C_\ast$; since the map $t \times_C t$ is precisely $t \circ \pi_2$, this equation is equivalent to the target equation $t \circ k = t \circ \pi_2$ for internal categories.

Finally, by unwinding the definitions of all the relevant maps in the same manner as done for the previous equations, it is straightforward to check that the coassociativity equation for $\delta$ on directions, namely that up to the associator isomorphism
\[
\left((\delta \trii \id_c) \circ \delta\right)^\sharp = \left((\id_c \trii \delta) \circ \delta\right)^\sharp,
\]
is equivalent to the associativity equation for internal categories.
\end{proof}

\begin{remark}
In $\poly_{\smset}$, \cref{comonoidcategory} recovers the result of \cite{ahman2016directed} that polynomial comonads agree with ordinary categories, because every function is exponentiable (i.e.\ counts as a polynomial $c$). 

However, note that the polynomial comonad perspective on categories differs from the usual definition of category, in terms of Hom-sets or graphs, in that it presents a category as a set of objects each equipped with a set of outgoing arrows. Not just the identity and composite arrows but also the target objects of the arrows themselves are treated as algebraic (or more accurately, coalgebraic) structure on the outgoing-arrow sets. 
\end{remark}

\begin{example}\label{discretecomonoid}
For any object $C$ in a finite limit category $\E$, the polynomial $C\yon$ carries a comonoid structure given by the cartesian morphisms $C\yon \to \yon$ and $C\yon \to C\yon \tri C\yon \cong (C \times C)\yon$ arising from the projection and diagonal maps. The corresponding internal category is ``discrete'' in the sense that the source, target, identity, and composition maps are all isomorphisms, and when $\E$ is the category of sets these comonoids are precisely the discrete categories.
\end{example}

\begin{example}
For a polynomial $p$, the comonoid $p_\ast$ resembles an ``indiscrete'' internal category on each fiber of $p$ separately. For polynomials in $\smset$, this category is precisely the disjoint union over $I \in P$ of the indiscrete categories on the sets $p[I]$.

The comonoid $c\coloneqq\lens{p}{p}$ in $\poly_{\smset}$ is the category whose objects are the positions of $p$ and whose morphisms $I \to J$ are given by functions $p[J] \to p[I]$ in $\smset$. Thus $c$ models the opposite of the full subcategory of $\smset$ spanned by the fibers $p[I]$ for all $I \in P$.
\end{example}

\begin{example}
In $\poly_{\smcat}$, comonoids are thus precisely the strict double categories whose source functor is Conduch\'e.\footnote{Recovering pseudo-double categories would require defining a 2-category structure on $\poly_{\smcat}$, which is beyond the scope of this paper.} While certainly not every double category has this property, it is the case for various double categories commonly studied: every equipment---double category in which the vertical arrows can be regarded as horizontal arrows in either direction---has this structure as does the double category commutative squares in a category. Concretely, the Conduch\'e-source condition states that for any factorization of the vertical source arrow of a square in the double category, there is a compatible vertical factorization of the square which is unique up to morphisms between such factorizations.
\end{example}

While the correspondence in \cref{comonoidcategory} demonstrates that polynomial comonoids and internal categories with exponentiable source map fundamentally contain the same information, the original result of \cite{ahman2016directed} (as summarized in \cite[Theorem 2.2.5]{spivak2021functorial}) is that not only do these structures agree on their data but also they agree on morphisms between them. However, the relevant morphisms between categories are not functors but rather cofunctors, whose internal analogue we now describe. Internal cofunctors in a category with pullbacks were first defined in \cite[Definition 12]{clarke2020internal}, and in this case they also specialize \cite[Definition 4.2.1]{aguiar1997internal}.

\begin{definition}
For categories $\C,\C'$ internal to $\E$, an internal \emph{cofunctor} $\phi$ from $\C$ to $\C'$ is an isomorphism class of commuting diagrams of the form in \eqref{eqn.cofunctor},%
\footnote{The isomorphisms referenced here are between choices of the pullback $C_0 \times_{C'_0} C'_1$.}
\begin{equation}\label{eqn.cofunctor}
\begin{tikzcd}
C_1 \dar[swap]{s} & C_0 \times_{C'_0} C'_1 \lar[swap]{\phi^\sharp} \rar \dar \ar[phantom]{dr}[pos=0]{\lrcorner} &[-10pt] C'_1 \dar{s'} \\
C_0 & C_0 \lar[equals] \rar{\phi_1} & C'_0
\end{tikzcd}
\end{equation}
such that the diagrams in \eqref{eqn.cofunctoreqs}, with notation as in \cref{def.internal_cat}, commute:
\begin{equation}\label{eqn.cofunctoreqs}
\begin{gathered}
\begin{tikzcd}
C_1 & C_0 \times_{C'_0} C'_1 \lar[swap]{\phi^\sharp} \\
C_0 \uar{i} \rar[swap]{\cong} & C_0 \times_{C'_0} C'_0 \uar[swap]{\id_{C_0} \times_{C'_0} i'}
\end{tikzcd}\qquad\quad\begin{tikzcd}
C_1 \dar[swap]{t} & C_0 \times_{C'_0} C'_1 \lar[swap]{\phi^\sharp} \rar &[-10pt] C'_1 \dar{t'} \\
C_0 & C_0 \lar[equals] \rar{\phi_1} & C'_0
\end{tikzcd}
\\
\begin{tikzcd}[column sep=45pt]
C_1 \times_{C_0} C_1 \dar[swap]{k} & C_1 \times_{C_0} (C_0 \times_{C'_0} C'_1) \lar[swap]{\id_{C_1} \times_{C_0} \phi^\sharp} \cong C_1 \times_{C'_0} C'_1 & C_0 \times_{C'_0} C'_1 \times_{C'_0} C'_1 \lar[swap]{\phi^\sharp \times_{C'_0} \id_{C'_1}} \dar{\id_{C_0} \times_{C'_0} k'} \\
C_1 & &  C_0 \times_{C'_0} C'_1 \ar{ll}{\phi^\sharp}  
\end{tikzcd}
\end{gathered}
\end{equation}
\end{definition}

The data of an internal cofunctor in \eqref{eqn.cofunctor} is clearly the same as that of a morphism as in \eqref{eqn.polymap} of the corresponding polynomials (from \cref{comonoidcategory}), but the equations in \eqref{eqn.cofunctoreqs} are more straightforward than the counit and comultiplication equations for polynomial comonoid homomorphisms when unwound into equations in the category $\E$, just as the equations governing internal categories are simpler to check than those of polynomial comonoids. If the two sets of equations are shown to be equivalent, identities and composites of internal cofunctors can be defined to be the same as identities and composites of comonoid homomorphisms, so that the bijection on objects from \cref{comonoidcategory} extends to an isomorphism of categories.

\begin{corollary}\label{cofunctorequivalence}
There is an isomorphism of categories between $\comon(\poly_\E)$ and a category whose objects are categories $\C$ internal to $\E$ with exponentiable source maps and whose morphisms are internal cofunctors. 
\end{corollary}

\begin{proof}
Let $(c,\epsilon,\delta)$ and $(c',\epsilon',\delta')$ be polynomial comonoids (or equivalently internal categories), and consider a morphism $\phi$ of polynomials from $c$ to $c'$, or equivalently the data of an internal cofunctor from $\C$ to $\C'$ as in \eqref{eqn.cofunctor} without assuming the equations in \eqref{eqn.cofunctoreqs}. We show that the comonoid homomorphism equations for $\phi$ are equivalent to the identity, target, and composition equations in \eqref{eqn.cofunctoreqs}, from which the result follows immediately. By \cref{preservereflect}, it suffices to restrict to the case when the category $\E$ is a presheaf category $\Ah$.

The counit preservation equation for $\phi$ states that $\epsilon$ agrees with the composite $\epsilon' \circ \phi$ of polynomial morphisms pictured on the left in \eqref{eqn.counits}.
\begin{equation}\label{eqn.counits}
\begin{tikzcd}[column sep=scriptsize]
C_\ast \dar[swap]{c} & C \times_{C'} C'_\ast \lar[swap]{\phi^\sharp} \rar \dar \ar[phantom]{dr}[pos=0]{\lrcorner} &[-10pt] C'_\ast \dar{c'} & C' \lar[swap]{(\epsilon')^\sharp} \rar \dar[equals] \ar[phantom]{dr}[pos=0]{\lrcorner} & 1 \dar[equals] \\
C & C \lar[equals] \rar{\phi_1} & C' & C' \lar[equals] \rar & 1
\end{tikzcd}\qquad\begin{tikzcd}
C_\ast \dar[swap]{c} &[40pt] C \times_{C'} C' \lar[swap]{\phi^\sharp \circ (\id_{C} \times_{C'} (\epsilon')^\sharp)} \rar \dar[swap]{\cong} \ar[phantom]{dr}[pos=0]{\lrcorner} & 1 \dar[equals] \\
C & C \lar[equals] \rar{\epsilon'_1 \circ \phi_1} & 1
\end{tikzcd}
\end{equation}
That this composite, pictured on the right in \eqref{eqn.counits}, agrees with $\epsilon$ is precisely equivalent to the commuting of the upper left diagram in \eqref{eqn.cofunctoreqs}, where $i=\epsilon^\sharp$ and $i'=(\epsilon')^\sharp$.

The comultiplication preservation equation for $\phi$ states that the composites $\delta' \circ \phi$ and $(\phi \trii \phi) \circ \delta$ of polynomial morphisms, pictured for an object $a$ in $\A$ on the left and right respectively in \eqref{eqn.comults}, agree.
\begin{equation}\label{eqn.comults}
\begin{tikzcd}
(C_\ast)_a \dar[swap]{c} &[-10pt]
(C \times_{C'} C'_\ast)_a 
\lar[swap,pos=.3]{\phi^\sharp} \rar[shorten=-2] \dar \ar[phantom]{dr}[pos=0]{\lrcorner} &[-15pt]
(C'_\ast)_a \dar{c'} &[-5pt] 
(C'_\ast \times_{C'} C'_\ast)_a 
\lar[swap,pos=.3]{(\delta')^\sharp} \rar[shorten >=-12] \dar \ar[phantom]{dr}[pos=0]{\lrcorner} & [-15pt]
\coprod\limits_{\substack{x' \in C'_a \\ f' \colon c'[x'] \to C'}} \coprod\limits_{y' \in c'[x']_a} \!\! c'[f'(y')]_a 
\dar[shorten <=-8,shift left=5]
\\
C_a & C_a \lar[equals] \rar{\phi_1} & C'_a & C'_a \lar[equals] \rar{\delta'_1} & 
\coprod\limits_{x' \in C'_a} \Hom_{\Ah}\left(c'[x'],C'\right) 
\end{tikzcd}
\end{equation}
\[\begin{tikzcd}
(C_\ast)_a \dar[swap]{c} &[-15pt]
(C_\ast \times_{C} C_\ast)_a 
\lar[shorten=-3,swap,pos=.3]{\delta^\sharp} \rar[shorten >=-13,shorten <=-3] \dar \ar[phantom]{dr}[pos=.05]{\lrcorner} &[-35pt]
\coprod\limits_{\substack{x \in C_a \\ f \colon c[x] \to C \\ y \in c[x]_a}} \!\!\!\!\!\!\! c[f(y)]_a 
\dar[shorten <=-29,shift left=7] &[-35pt]
\coprod\limits_{\substack{x \in C_a \\ f \colon c[x] \to C \\ y' \in c'[\phi_1(x)]_a}} \!\!\!\!\!\!\! c'[\phi_1(f(\phi^\sharp(x,y')))]_a 
\lar[shorten <=-19,shorten >=-4,swap,pos=-2.4]{(\phi \trii \phi)^\sharp} \rar[shorten >=-15] \dar[shorten <=-29,shift left=4] \ar[phantom]{dr}[pos=-1.3]{\lrcorner} &[-30pt]
\coprod\limits_{\substack{x' \in C'_a \\ f' \colon c'[x'] \to C' \\ y' \in c'[x']_a}} \!\!\!\!\!\!\! c'[f'(y')]_a \dar[shorten <=-29,shift left=8]
\\
C_a & C_a \lar[equals] \rar[shorten >=-5]{\delta_1} & 
\coprod\limits_{x \in C_a} \!\! \Hom_{\Ah}\left(c[x],C\right) & 
\coprod\limits_{x \in C_a} \!\! \Hom_{\Ah}\left(c[x],C\right) \lar[equals, shorten=-3] \rar[shorten >=-8, inner sep=10]{(\phi \trii \phi)_1} & 
\coprod\limits_{x' \in C'_a} \!\! \Hom_{\Ah}\left(c'[x'],C'\right) 
\end{tikzcd}\]

On positions, this states that the maps sending an element $x \in C_a$ to respectively
\[
\left(\phi_1(x), \delta'_1(\phi_1(x))\right) \qqand 
\left(\phi_1(x),\phi_1 \circ \delta_1(x) \circ \phi^\sharp\restrict{c'[\phi_1(x)]}\right)
\]
agree, where $\phi^\sharp\restrict{c'[\phi_1(x)]}$ is the restriction of $\phi^\sharp$ to a map 
\[
c'[\phi_1(x)] \cong \yon(a) \pb{\phi_1(x)}{c'} C'_\ast \to \yon(a) \pb{x}{c} C_\ast \cong c[x].
\]
Under the correspondence between $\delta_1,\delta'_1$ and the target maps $t,t'$ of the internal categories associated to $c,c'$, this equation shows that for any pair $(x,y') \in (C \times_{C'} C'_\ast)_a$, where by definition $y' \in c'[\phi_1(x)]_a$, this equation is equivalent to
\[
t'(y) = \phi_1(t(\phi^\sharp(x,y))),
\]
which is precisely the target equation for internal cofunctors in \eqref{eqn.cofunctoreqs}.

For the comultiplication equation for $\phi$ directions, first observe that both composite morphisms of polynomials in \eqref{eqn.comults} have the form of \eqref{eqn.comultcomp},
\begin{equation}\label{eqn.comultcomp}
\begin{tikzcd}
(C_\ast)_a \dar[swap]{c} &
(C \times_{C'} C'_\ast \times_{C'} C'_\ast)_a 
\lar[swap]{?} \rar[shorten >=-5] \dar \ar[phantom]{dr}[pos=0]{\lrcorner} &
\coprod\limits_{\substack{x' \in C'_a \\ f' \colon c'[x'] \to C'}} \coprod\limits_{y' \in c'[x']_a} \!\! c'[f'(y')]_a 
\dar[shorten <=-8,shift left=5]
\\
C_a & C_a \lar[equals] \rar{\delta'_1 \circ \phi_1} & 
\coprod\limits_{x' \in C'_a} \Hom_{\Ah}\left(c'[x'],C'\right) 
\end{tikzcd}
\end{equation}
in the first case simply by inspection of the diagram and in the second case using the comultiplication equation on positions discussed above. Given this, the comultiplication equation for $\phi$ on directions then states that the maps sending a triple $(x,y',z') \in (C \times_{C'} C'_\ast \times_{C'} C'_\ast)_a$ to respectively
\[
\phi^\sharp(x,(\delta')^\sharp(y',z')) \qqand \delta^\sharp(\phi^\sharp(x,y'),\phi^\sharp(\delta_1(x)(y'),z'))
\]
agree. In the language of the corresponding internal categories, the equivalent equation is 
\[
\phi^\sharp(x,k'(y',z')) = k(\phi^\sharp(x,y'),\phi^\sharp(t(y'),z')),
\]
which is precisely the composition equation for internal cofunctors in \eqref{eqn.cofunctoreqs}.
\end{proof}

\begin{example}
In $\poly_{\smcat}$, internal cofunctors between double categories (with exponentiable source functor) are maps given by compatible functors in the vertical direction and cofunctors in the horizontal direction. More precisely, a comonoid homomorphism $\phi \colon \C \to \D$ in $\poly_{\smcat}$ consists of a functor $\phi_1 \colon C_0 \to D_0$ between the vertical categories, a cofunctor from the horizontal category of $\C$ to that of $\D$, and a cofunctor from the category of vertical arrows and squares of $\C$ to that of $\D$ such that these cofunctors agree with $\phi_1$ on objects and on morphisms commute with vertical sources, targets, identities, and composites.
\end{example}

\section{Coalgebras as internal copresheaves}\label{sec.coalgebras}

Having shown that $\tri$-comonoids in the monoidal category $\poly_\E$ are precisely the internal categories in $\E$ whose source morphism is exponentiable, we now turn to studying additional categorical structures relating to comonoids. To start, we show that coalgebras for the comonads on $\E$ induced by polynomial comonoids recover the usual internal analogue of copresheaves (also known as internal diagrams) on an internal category \cite[Section V.7]{macLane1992sheaves}. Recall the notation $\P \colon \poly_\E \to \StrFunn(\E,\E)$ from \cref{polynomialfunctor,Pfunctor}.

\begin{definition}
Given a $\tri$-comonoid $c$ in $\poly_\E$ for a finite limit category $\E$, a $c$-coalgebra is an object $S$ in $\E$ equipped with a morphism $\kappa \colon S \to \P(c)(S)$ such that the counit and comultiplication diagrams in \eqref{eqn.coalgeqs} commute.
\begin{equation}\label{eqn.coalgeqs}
\begin{tikzcd}
S \rar{\kappa} \ar{dr}[swap]{\cong} & \P(c)(S) \dar{\P(\epsilon)_S} \\
& \P(\yon)(S)
\end{tikzcd}\qquad\quad\begin{tikzcd}[row sep=tiny]
& \P(c)(S) \rar{\P(\delta)_S} &[15pt] \P(c \trii c)(S) \ar{dd}{\cong} \\
S \ar{ur}{\kappa} \ar{dr}[swap]{\kappa} \\
& \P(c)(S) \rar[swap]{P(c)(\kappa)} & \P(c)(\P(c)(S))
\end{tikzcd}
\end{equation}
A $c$-coalgebra homomorphism $S \to S'$ is a morphism in $\E$ which commutes with $\kappa,\kappa'$. We write $c\coalg$ for the category of $c$-coalgebras and homomorphisms.
\end{definition}

Recall that by the universal property of $\P(c)(S)$, a morphism $S \to \P(c)(S)$ in $\E$ corresponds to morphisms $S \to C$ and $S \times_C C_\ast \to S$

\begin{example}
In $\poly_{\smset}$, a function $S \to C$ labels each element of $S$ with an object of $C$, and a function $S \times_C C_\ast \to S$ assigns to each element of $S$ and each morphism out of its label in $C$ a new element of $S$ labeled by the codomain of that morphism. This codomain condition is a consequence of the comultiplication equation, which also ensures that these assignments are functorial in $c$ while the counit equation ensures that the identity morphisms in $c$ act as identities on $S$. A coalgebra structure therefore exhibits $S$ as the total set of a functor from the category $c$ to the category of sets, and any copresheaf on $c$ conversely produces a coalgebra in this manner. The goal of this section is to generalize this correspondence to coalgebras in $\poly_\E$ and internal copresheaves for any finite limit category $\E$.
\end{example}

When working with ordinary categories, copresheaves (as in, functors to $\smset$) are a frequent object of attention, suggesting that there ought to be an internal notion of copresheaf. However, there is not generally an analogue of the category of sets for categories internal to $\E$, so in order to imitate working with copresheaves we must instead generalize the equivalent notion of discrete opfibration using a generalization to the internal setting first given in \cite[Examples 10, 13]{clarke2020internal}. Recall that for a category $\A$, the category of functors $\A \to \smset$ is equivalent to the category of discrete opfibrations into $\A$.

\begin{definition}
For $\C$ a category internal to $\E$, an internal copresheaf (or internal discrete opfibration) on $\C$ is an internal cofunctor $\phi \colon \D \to \C$ whose underlying morphism of polynomials is cartesian. A morphism of internal copresheaves $\D,\D'$ on $\C$ is an internal cofunctor $\D \to \D'$ commuting over $\C$.
\end{definition}

\begin{remark}
Just as among ordinary categories, discrete opfibrations between them can be equivalently regarded as either functors or cofunctors. Furthermore, maps between internal copresheaves on $\C$ can also be equivalently defined as either functors or cofunctors commuting over $\C$, as in both cases the commutativity condition forces such a functor or cofunctor to be a discrete opfibration.
\end{remark}

\begin{theorem}\label{coalgcopresheaves}
For $c$ a comonoid in $\poly_\E$ for $\E$ a finite limit category, the category $c\coalg$ of $c$-coalgebras is equivalent to the category of internal copresheaves on $c$.
\end{theorem}

\begin{proof}
Recall from \cref{polynomialfunctor} that for $S$ an object of $\E$, $\P(c)(S)$ is defined as $\Pi_c (S \times C_\ast)$. By the universal property of $\P(c)$, a map $\kappa \colon S \to \P(c)(S)$ is uniquely determined by maps $\kappa_1 \colon S \to C$ and $\kappa^\sharp \colon S \times_C C_\ast \to S$. The projection morphism $\pi_1 \colon S \times_C C_\ast \to S$ is a polynomial (as a pullback of $c$) with a cartesian morphism $\overline\kappa$ to $c$, which we further show forms an internal category $\mathbf{S}$ with $\kappa^\sharp$ as the target map and $\overline\kappa$ an internal discrete opfibration $\mathbf{S} \to \C$.

The map $S \times_C C_\ast \from S$ which chooses identities is the morphism into the pullback induced by $\id_S$ and $i \colon C_\ast \from C$, while the composition map 
\[
(S \times_C C_\ast) \times_S (S \times_C C_\ast) \to S \times_C C_\ast
\]
is induced by the projection
\[
(S \times_C C_\ast) \times_S (S \times_C C_\ast) \To{\pi_1} S \times_C C_\ast \To{\pi_1} S
\]
and the composite 
\[
(S \times_C C_\ast) \times_S (S \times_C C_\ast) \to C_\ast \times_C C_\ast \To{k} C_\ast,
\]
which by construction commute over $C$. The equations for an internal category and internal cofunctor are then straightforward to check by the construction of $\mathbf{S}$ from the internal category $\C$.

Conversely, let $\phi \colon \D \to \C$ be an internal discrete opfibration in $\E$. The maps $\phi_1 \colon D \to C$ and $t \colon D_\ast \cong D \times_C C_\ast \to D$ determine a map $D \to \P(c)(D)$, and it is straightforward to check that it satisfies the coalgebra equations, since $\phi$ is a comonoid homomorphism. This assignment is evidently inverse up to isomorphism to the construction above of an internal discrete opfibration from a coalgebra, so it remains only to show that the morphisms in the two categories agree.

A morphism of $c$-coalgebras is a morphism $\psi$ making the square in \eqref{eqn.coalgmap} commute,
\begin{equation}\label{eqn.coalgmap}
\begin{tikzcd}[column sep=large]
S \rar{\psi} \dar[swap]{\kappa} & S' \dar{\kappa'} \\
\P(c)(S) \rar[swap]{\P(c)(\psi)} & \P(c)(S')
\end{tikzcd}
\end{equation}
which under the universal properties of $\Pi$ and products corresponds to the equations $\kappa_1 \circ \psi = \kappa'_1$ and $\kappa^\sharp \circ (\psi \times_C C_\ast) = \psi \circ (\kappa')^\sharp$. In the corresponding internal categories, as the identities and compositions are derived from those of $\C$ using $\kappa,\kappa'$, this data is equivalent to the pullback square 
\[
\begin{tikzcd}
S \times_C C_\ast \rar{\pi_1} \dar[swap]{\psi \times_C C_\ast} & S \dar{\psi} \\
S' \times_C C_\ast \rar{\pi_1} & S' 
\end{tikzcd}
\]
satisfying the internal cofunctor equations. Since all cofunctors between these internal categories commuting over $\C$ are of this form, the proof is complete.
\end{proof}

\begin{example}
If $\E$ is additionally cartesian closed, then the morphism $S \to 1$ is a polynomial which we refer to as $\yon^S$, and the comonoid $(\yon^S)_\ast$ is precisely the polynomial $S\yon^S$. In this case, it is straightforward to check that a $c$-coalgebra structure on $S$ is precisely the data of a cofunctor $S\yon^S \to c$, both consisting of morphisms $S \to C$ and $S \times_C C_\ast \to S$ satisfying analogous counit and comultiplication equations. In this case, the internal discrete opfibration corresponding to the coalgebra $S$ is simply the cartesian component of this cofunctor.
\end{example}

\begin{example}\label{doublecopresheaves}
In $\smcat$, an internal discrete opfibration $\phi \colon \D \to \C$ consists of a pullback square from $s' \colon D_1 \to D_0$ to $s \colon C_1 \to C_0$ commuting with target, identity, and composition functors. This is in particular a double functor which is a discrete opfibration on both the horizontal categories and the categories of vertical arrows and squares. We denote these categories by $C^\dagger_0$ and $C^\dagger_1$, as they are the categories of objects and horizontal morphisms respectively in the \emph{transpose} of $\C$, where horizontal and vertical morphisms are swapped.\footnote{While it is beyond the scope of this paper, an analogue of the following discussion for pseudo-double categories is particularly challenging as the transpose operation is not available.}

This makes the transpose double functor $\phi^\dagger \colon \D^\dagger \to \C^\dagger$ a \emph{discrete double opfibration} in the sense of \cite[Definition 2.2.6]{lambert2021discrete}.%
\footnote{While Lambert defines discrete double fibrations rather than opfibrations, the analogy is entirely straightforward.} 
Discrete double opfibrations into a double category are equivalent to lax double functors from that double category to $\sspan$, the double category of sets, functions, spans, and morphisms of spans by \cite[Theorem 2.4.3]{lambert2021discrete}, so we have that a $c$-coalgebra in $\poly_{\smcat}$ corresponds to a lax double functor $\C^\dagger \to \sspan$, the standard (\cite{pare2011yoneda}) notion of a double copresheaf on the transpose of the double category $\C$. In particular, this lax double functor sends objects $a$ in $\C$ to the set $\phi^{-1}(a)$ of objects over $a$ in $\D$, vertical arrows $f \colon a \to b$ in $\C$ to the span $\phi^{-1}(a) \from \phi^{-1}(f) \to \phi^{-1}(b)$, and horizontal arrows/squares to the functions they correspond to under the discrete opfibrations of $\phi^\dagger$. The identitor and compositor structure maps arise from vertical identities and composition in $\D$.
\end{example}

\section{Bicomodules: typed polynomials and familial functors}

Algebraic structures often have many different types of morphisms between them, typically including some sort of structure-preserving maps as well as a notion of an object which interacts in a certain way with both the domain and codomain, such as a span or bimodule. For comonoids in $\poly_\E$, a natural choice is that of a bicomodule, which as we show in \cref{typedpolycompose,rem.pra}, recovers both the typed polynomials of \cite{kock2012polynomial,weber2015polynomials} and parametric right adjoint functors between copresheaf categories.

\begin{definition}
Given a $\tri$-comonoid $c$ in $\poly_\E$ for a finite limit category $\E$, a \emph{left $c$-comodule} is a polynomial $m$ equipped with a morphism of polynomials $\kappa \colon m \to c \trii m$ (the \emph{left coaction}) such that the counit and comultiplication diagrams in \eqref{eqn.comodeqs} commute.
\begin{equation}\label{eqn.comodeqs}
\begin{tikzcd}
m \rar{\kappa} \ar{dr}[swap]{\cong} & c \trii m \dar{\epsilon \trii \id_m} \\
& \yon \trii m
\end{tikzcd}\qquad\quad\begin{tikzcd}[row sep=tiny]
& c \trii m \rar{\delta \trii \id_m} &[15pt] (c \trii c) \trii m \ar{dd}{\cong} \\
m \ar{ur}{\kappa} \ar{dr}[swap]{\kappa} \\
& c \trii m \rar[swap]{\id_c \trii \kappa} & c \trii (c \trii m)
\end{tikzcd}
\end{equation}
A \emph{right $c$-comodule} is a polynomial $m$ equipped with a \emph{right coaction} morphism of the form $\chi \colon m \to m \trii c$ satisfying analogous equations to those in \eqref{eqn.comodeqs}. A (left or right) $c$-comodule homomorphism $m \to m'$ is a morphism of the underlying polynomials which commutes with the respective coaction maps. 
\end{definition}

\begin{remark}
By transpose along the right coclosure, when it exists, a left $c$-comodule structure on $m$ corresponds to a cofunctor $\lens{m}{m} \to c$. Similarly, a right $c$-comodule structure on $m$ corresponds to a cofunctor $m_\ast \to c$. In both cases the counit and comultiplication equations for comonoid homomorphisms and comodules correspond exactly.
\end{remark}

\begin{remark}\label{comodcoalg}
As the positions of $c \tri m$ are given by $\P(c)(M)$ and the equations are analogous, a left $c$-comodule structure on $m$ endows $M$ with the structure of a $c$-coalgebra. Meanwhile, a right $d$-comodule structure on $m$ in fact endows $M_\ast$ with the structure of a $d$-coalgebra, as a morphism $m \to m \tri d$ satisfying the comodule equations consists of morphisms $M_\ast \to D$ and $M_\ast \times_D D_\ast \to M_\ast$ satisfying the coalgebra equations.
\end{remark}

\begin{definition}
For $\tri$-comonoids $c,d$ in $\poly_\E$ for a finite limit category $\E$, a $(c,d)$-bicomodule is a polynomial $m$ equipped with a left $c$-comodule structure $\kappa$ and a right $d$-comodule structure $\chi$ such that the diagram in \eqref{eqn.bicomod} commutes.
\begin{equation}\label{eqn.bicomod}
\begin{tikzcd}
c \trii m \dar[swap]{\id_c \trii \chi} & m \lar[swap]{\kappa} \rar{\chi} & m \trii d \dar{\kappa \trii \id_d} \\
c \trii (m \trii d) & & (c \trii m) \trii d \ar{ll}{\cong}
\end{tikzcd}
\end{equation}
Given cofunctors $\phi \colon c \to c'$ and $\psi \colon d \to d'$, a $(c,d)$-bicomodule $m$ and a $(c',d')$-bicomodule $m'$, a $(\phi,\psi)$-homomorphism of bicomodules $\gamma \colon m \to m'$ is a morphism of the underlying polynomials which makes the squares in \eqref{eqn.bicomodhom} commute. 
\begin{equation}\label{eqn.bicomodhom}
\begin{tikzcd}
c \trii m \dar[swap]{\phi \trii \gamma} & m \lar[swap]{\kappa} \rar{\chi} \dar{\gamma} & m \trii d \dar{\gamma \trii \psi} \\
c' \trii m' & m' \lar{\kappa'} \rar[swap]{\chi'} & m' \trii d' 
\end{tikzcd}
\end{equation}
We write $\bic{}{}(\poly_\E)$ for the category whose objects are triples $(c,d,m)$ where $m$ is a $(c,d)$-bicomodule in $\poly_\E$, and whose morphisms are triples $(\phi,\psi,\gamma)$ where$\gamma$ is a $(\phi,\psi)$-homomorphism of bicomodules.
\end{definition}

\begin{example}\label{setbicomod}
For ordinary categories $c,d$, a $(c,d)$-bicomodule $m$ in $\poly_{\smset}$ contains precisely the data of a parametric right adjoint functor from the category of $d$-copresheaves to that of $c$-copresheaves (\cite[Theorem 2.3.1]{spivak2021functorial}, originally due to \href{https://www.youtube.com/watch?v=tW6HYnqn6eI}{Richard Garner}).

By \cref{comodcoalg}, the set $M$ carries a $c$-copresheaf structure and the set $M_\ast$ carries a $d$-copresheaf structure. Further inspection shows that for $x \in M_I$, the preimage of $I \in C$ along the function $M \to C$, and a morphism $f \colon I \to J$ in $C$, there is a function $m[f_\ast(x)] \to m[x]$ functorial in $f$, and that by the compatibility equation in \eqref{eqn.bicomodhom} these functions are morphisms of $d$-copresheaves. The functor associated to $m$ then sends a $d$-copresheaf $X$ to the $c$-copresheaf whose component at an object $I$ of $c$ is given by
\[
\coprod_{x \in M_I} \Hom_{d\coalg}(m[x],X)
\] 
and whose structure map for $f \colon I \to J$ in $c$ derive from the functions $M_I \to M_J$ and $m[f_\ast(x)] \to m[x]$ associated to $m$.
\end{example}

\begin{example}\label{typedpolybicomod}
Consider now the discrete comonoids $C\yon,D\yon$ of \cref{discretecomonoid} for objects $C,D$ in $\E$. For a polynomial $m$, the composite $C\yon \trii m$ is given by the scalar product $Cm$ (\cref{scalarmult}), and a left $C\yon$-comodule structure on $m$ is uniquely determined by a morphism $M \to C$ in $\E$, as by the counit equation for modules any coaction map $m \to Cm$ is cartesian and the map on positions $M \to C \times M$ is a section of the second projection map.

A right $D\yon$-coaction on $m$ on positions consists of a map $M \to \Pi_m (D \times M_\ast)$ induced by the identity morphism on $M$ and a morphism $M_\ast \to D$. The coaction on directions has the form $M_\ast \to M_\ast$ and is again forced to be the identity by the counit equation for comodules.

A $(C\yon,D\yon)$-bicomodule therefore consists of precisely a polynomial $m$ along with morphisms $M \to C$ and $M_\ast \to D$, where the compatibility condition holds automatically. This is precisely the data of a typed polynomial from $D$ to $C$. Furthermore, as cofunctors of the form $\phi \colon C\yon \to C'\yon$ and $\psi \colon D\yon \to D'\yon$ are precisely morphisms $C \to C'$ and $D \to D'$ respectively in $\E$, a homomorphism of bicomodules $m \to m'$ is precisely a morphism of the corresponding typed polynomials, namely, a morphism of polynomials $m \to m'$ whose component on positions commutes with $\phi$ and whose component on directions commutes with $\psi$.
\end{example}

In \cref{typedpolycompose} we further show that composition of bicomodules agrees with composition of typed polynomials as defined in \cite[Definition 1.11]{kock2012polynomial}. In this way, our presentation of the theory of generalized polynomials recovers the treatment of \cite{kock2012polynomial,weber2015polynomials} in terms of typed polynomials.

In fact, we recover the entire double category of typed polynomials of \cite{kock2012polynomial} as a full double subcategory of a double category $\ccatsharp_\E$ whose objects are comonoids in $\poly_\E$, vertical morphisms are cofunctors, horizontal morphisms are bicomodules, and squares are a suitable generalization of maps between bicomodules. To this end, we define a general notion of bicomodule composition which, in the case of bicomodules between discrete comonoids, will always exist for a finite limit category $\E$.

\begin{definition}
For a $(c,d)$-bicomodule $m$ and a $(d,e)$-bicomodule $m'$, their composite $m \tri_d m'$ is defined as the equalizer of the diagram
\[
m \tri m' \Tto[35pt]{\chi \trii \id_{m'}}{\id_m \trii \kappa'} m \tri d \tri m'
\]
in $\poly_\E$, if such an equalizer exists and is preserved by the functor $- \tri e$. As the functors $c \tri -$ (by \cref{trilimit}) and $- \tri e$ (by assumption) preserve this equalizer, the coactions $\kappa \colon m \to c \tri m$ and $\chi' \colon m' \to m' \tri e$ induce a $(c,e)$-bicomodule structure
\[
c \tri (m \tri_d m') \from m \tri_d m' \to (m \tri_d m') \tri e
\]
on this composite, whose equations are tedious but straightforward to check. 
\end{definition}

For bicomodules to always be able to compose, there would need to be arbitrary equalizers in $\poly_\E$ preserved by $\tri$, which would require additional assumptions on the category $\E$ (for instance having coequalizers) that are beyond the scope of this paper. However, as we will see, in limited circumstances bicomodules can be guaranteed to compose.

\begin{example}\label{idbicomod}
For any comonoid $c$ in $\poly_\E$, there is an identity $(c,c)$-bicomodule given by the polynomial $c$ and the comultiplication $\delta \colon c \to c \tri c$ as both coactions. When $c$ is a discrete comonoid $C\yon$ the comultiplication is given by the diagonal map $C \to C \times C$, and this bicomodule corresponds to the identity typed polynomial on the object $C$ whose component morphisms are all $\id_C$.
\end{example}

\begin{lemma}\label{cartesianbicomod}
For comonoids $c,d,e$ in $\poly_\E$ for a finite limit category $\E$, a $(c,d)$-bicomodule $(m,\kappa,\chi)$ such that $\chi$ is a cartesian morphism, and a $(d,e)$-bicomodule $(m',\kappa',\chi')$ such that $\kappa'$ is a cartesian morphism, the composite $(c,e)$-bicomodule $m \tri_d m'$ exists.
\end{lemma}

\begin{proof}
This follows immediately from \cref{trilimit}, as the equalizer defining $m \tri_d m'$ consists only of cartesian morphisms and therefore exists and is preserved by $- \tri e$. 
\end{proof}

We are now ready to show that bicomodules in $\poly_\E$ allow us to recover the pseudo-double category of typed polynomials described in \cite{kock2012polynomial}.

\begin{theorem}\label{typedpolycompose}
There is a pseudo-double category $\ccatsharp_{\E,\disc}$ of discrete comonoids, homomorphisms, bicomodules, and bicomodule homomorphisms in $\poly_\E$ which is double-equivalent to the pseudo-double category of objects, morphisms, typed polynomials, and typed morphisms of polynomials in $\E$.
\end{theorem}

In particular, this pseudo-double category is given by the diagram 
\[
\bic{}{}(\poly_\E)_{\disc} \tto \comon(\poly_\E)_{\disc}
\]
in $\lgcat$, where the two functors send a bicomodule to its source and target comonoid respectively.

\begin{proof}
First observe that the techniques used in the construction of the monoidal category $\poly_\E$ in \cref{chap.poly} apply equally well to recovering the pseudo-double category of typed polynomials in $\E$ from that of typed polynomials in $\Ah$ established by \cite{kock2012polynomial}. 
In particular, the identity horizontal arrow on an object $C$ of $\E$ is given by the identity morphism on $C$, and the composite of two typed polynomials $p$ from $D$ to $C$ and $q$ from $E$ to $D$ is given by the composition of the top row of morphisms in \eqref{eqn.typedcompose}.
\begin{equation}\label{eqn.typedcompose}
\begin{tikzcd}
&[-55pt] (\Pi_p (Q \times_D P_\ast) \times_P P_\ast) \times_Q Q_\ast \ar{dd} \ar{rr} \ar[phantom]{ddrr}[pos=0]{\lrcorner} &[-30pt] &[-25pt] \Pi_p (Q \times_D P_\ast) \times_P P_\ast \ar{rr} \dar \ar[phantom]{ddrr}[pos=0]{\distpb} &[-25pt] & \Pi_p (Q \times_D P_\ast) \ar{dd} &[-25pt] {} \\
& & & Q \times_D P_\ast \ar{dl} \ar{dr} \\
& Q_\ast \rar{q} \ar{dl}[swap]{f} & Q \ar{dr}{g} & {} & P_\ast \rar{p} \ar{dl}[swap]{h} & P \ar{dr}{k} \\
E & & & D & & & C
\end{tikzcd}
\end{equation}

As described in Examples \ref{discretecomonoid} and \ref{typedpolybicomod}, the diagrams
\[\bic{}{}(\poly_\E)_{\disc} \tto \comon(\poly_\E)_{\disc}
\qqand
\poly_\E^{\tn{typed}} \tto \E
\]
are equivalent, and their identities agree by \cref{idbicomod}, so it remains only to show that the two notions of composition agree. In fact, by an argument analogous to the definition of the action of $\tri$ on morphisms in the proof \cref{monoidalestablish}, it suffices to show that the two definitions of composition agree on horizontal arrows, as given this the horizontal composition of squares can be defined via the equivalence of categories $\bic{}{}(\poly_\E)_{\disc} \simeq \poly_\E^{\tn{typed}}$.

Consider typed polynomials $p,q$ as above, and regard them as $(C\yon,D\yon)$- and $(D\yon,E\yon)$-bicomodules respectively. These bicomodules have a composite by \cref{cartesianbicomod} as the coactions of $p,q$ are automatically cartesian by \cref{typedpolybicomod}, so it suffices to show that the typed polynomial on the top row in \eqref{eqn.typedcompose} is the equalizer of the diagram
\[
p \tri q \Tto[35pt]{g \trii \id_q}{\id_p \trii h} p \tri D\yon \tri q \cong p \tri Dq.
\]
By the construction of connected limits of diagrams of cartesian morphisms in \cref{hasconnectedlimits}, it further suffices to define an isomorphism 
\[
\Pi_p (Q \times_D P_\ast) \cong 
\lim\left(\begin{tikzcd}
\Pi_p (Q \times P_\ast) 
\rar[shift left=2]{{\Pi_p(\langle \id_Q,g \rangle \times \id_{P_\ast})}}
\rar[shift right=2,swap]{{\Pi_p(\id_Q \times \langle h,\id_{P_\ast} \rangle)}}
&[50pt] \Pi_p (Q \times D \times P_\ast)
\end{tikzcd}\right)
\]
between these polynomials on positions. This isomorphism follows from the fact that $\Pi_p$ preserves limits as a right adjoint, and the general fact that in any finite limit category there is an isomorphism
\[
Q \times_D P_\ast \cong 
\lim\left(\begin{tikzcd}
Q \times P_\ast 
\rar[shift left=2]{{\langle \id_Q,g \rangle \times \id_{P_\ast}}}
\rar[shift right=2,swap]{{\id_Q \times \langle h,\id_{P_\ast} \rangle}}
&[40pt] Q \times D \times P_\ast
\end{tikzcd}\right)
\qedhere
\]
\end{proof}

\begin{remark}
In \cite[Corollary 2.1.10]{spivak2021functorial}, a pseudo-double category (in fact, an equipment) of all bicomodules in $\poly_{\smset}$ is constructed using the $\bic{}{}$ construction from \cite[Theorem 11.5]{shulman2008framed}. This construction requires that $\poly_\E$ has all equalizers and that they are preserved by $\tri$, which would require additional assumptions on the category $\E$ which are beyond the scope of this paper but which do in fact hold in many examples of interest. 
\end{remark}

Finally, just as bicomodules in $\poly$ correspond to parametric right adjoint functors between the associated copresheaf categories (\cref{setbicomod}), bicomodules in $\poly_\E$ induce functors between categories of internal copresheaves. 

\begin{theorem}
For comonoids $c,d$ in $\poly_\E$ for a finite limit category $\E$, a $(c,d)$-bicomodule $(m,\kappa,\chi)$ induces a connected-limit-preserving functor from $d\coalg$ to $c\coalg$.
\end{theorem}

\begin{proof}
This functor is defined similarly to composition of bicomodules; for a $d$-coalgebra $\kappa' \colon S \to \P(d)(S)$, define $m \tri_d S$ as the equalizer of the diagram
\[
\P(m)(S) \Tto[40pt]{\P(\chi)(S)}{\P(m)(\kappa')} \P(m \tri d)(S) \cong \P(m)(\P(d)(S)).
\]
The functor $\P(c)$ preserves connected limits, so using \cref{Pfunctor}, $\P(c)(m \tri_d S)$ is the equalizer of the diagram 
\[
\P(c \tri m)(S) \Tto[53pt]{\P(\id_c \tri \chi)(S)}{\P(c \tri m)(\kappa')} \P(c \tri m \tri d)(S).
\]
The morphism $\kappa \colon m \to c \tri m$ therefore induces a natural transformation from the first diagram to the second, and therefore a morphism $m \tri_D S \to \P(c)(m \tri_D S)$. Just like composition of bicomodules, the coalgebra equations for $m \tri_D S$ follow straightforwardly from those for $S$ and the bicomodule equations for $m$.

To see that this functor preserves connected limits, note that $\P(m)$ and $\P(d)$ preserve connected limits, as do equalizers.
\end{proof}

\begin{remark}\label{rem.pra}
While we do not prove that this functor is a parametric right adjoint in general, for sufficiently nice choices of the finite limit category $\E$ (such as presheaf categories) connected limit preserving functors are precisely the parametric right adjoints. Moreover, we also do not show that all parametric right adjoints between these copresheaf categories arise in this fashion, which we have come to expect is not the case without further assumptions. Were this to be true, with composition of bicomodules agreeing with composition of parametric right adjoints as expected, then given any two adjacent bicomodules one could compose them by composing their associated parametric right adjoint functors and extracting the corresponding bicomodule. This would suggest that bicomodule composition always exists, which as discussed above ought to require additional assumptions on the category $\E$.
\end{remark}

\begin{example}
One avenue for future work could be to work out the details and applications of bicomodules in $\poly_{\smcat}$, which should correspond to parametric right adjoint double functors between double categories of double copresheaves. Taking the perspective of \cite{spivak2012functorial} that parametric right adjoints describe ``data migration functors'' between copresheaf categories regarded as categories of database instances for a particular schema, this could contribute to the development of double categorical database theory as discussed in \cite{Lambert:blog}.
\end{example}

\printbibliography
\end{document}